\documentclass[a4paper,oneside,portrait,11pt, reqno]{AmsArt}

\usepackage[utf8]{inputenc} 

\usepackage[margin=1in]{geometry}

\usepackage[ps,all,arc,rotate]{xy}

\usepackage{amsfonts,amscd,amsthm,amsgen,amsmath,amssymb}
\usepackage[all]{xy}
\usepackage[vcentermath]{youngtab}
\usepackage {graphicx}
\usepackage{epstopdf}
\usepackage{latexsym}
\usepackage{epsfig,color}
\usepackage{caption}
\usepackage{subcaption}
\usepackage{wrapfig}

\usepackage{hyperref}
\hypersetup{
  bookmarks=true,
}

\def\includegraphics{}

\newtheorem{definition}[]{Definition} 
\newtheorem{theorem}[definition]{Theorem}
\newtheorem{corollary}[definition]{Corollary}
\newtheorem{proposition}[definition]{Proposition}
\newtheorem{remark}[definition]{Remark}
\newtheorem{lemma}[definition]{Lemma}
\theoremstyle{definition}

\newcommand{\noi}{\noindent}
\newcommand{\SL}{\mathrm{SL}}
\newcommand{\SO}{\mathrm{SO}}

\def\CI{{\mathcal I}}

\def\CL{{\mathcal L}}
\def\CM{{\mathcal M}}
\def\CMt{{\mathcal {\tilde M}}}

\def\CO{{\mathcal O}}

\DeclareMathOperator{\Prym}{Prym}

\newcommand{\C}{\mathbb{C}}
\newcommand{\Z}{\mathbb{Z}}
\newcommand{\R}{\mathbb{R}}

\date{\today}

\begin{document}

%
%


\title{Higgs bundles and exceptional isogenies}
\title{Higgs bundles and exceptional isogenies}
 \author{Steven B. Bradlow and Laura P. Schaposnik}

\address{Department of Mathematics, University of Illinois, Urbana, IL 61801, USA}
\email{bradlow@illinois.edu}
\email{schaposnik@illinois.edu}

%
%
%
%


%
\maketitle
\begin{abstract}We explore relations between Higgs bundles that result from isogenies between low-dimensional Lie groups, with special attention to the spectral data for the Higgs bundles. We focus on isogenies onto $\SO(4,\C)$ and $\SO(6,\C)$ and their split real forms.  Using fiber products of spectral curves, we obtain directly the desingularizations of the (necessarily singular) spectral curves associated to orthogonal Higgs bundles. In the case of $\SO(6,\C)$ our construction can be interpreted as a new description of Recillas' trigonal construction.
\end{abstract}
%

 Exceptional isomorphisms between low rank Lie algebras, and the corresponding isogenies between Lie groups, have been a source of fascination since first catalogued by Cartan a century ago (\cite[p.~352-355]{Cartan},  {see} also \cite[p. 519]{helga}). In this paper we explore the implications for Higgs bundles   of the  isogenies between Lie groups of rank 2 and 3: 
 \begin{align}\label{isogenies}
\mathcal{I}_2:\SL(2,\C)\times\SL(2,\C)\rightarrow\SO(4,\C);\\
\mathcal{I}_3:\SL(4,\C)\rightarrow\SO(6,\C)\label{isogenies3};
\end{align} 

 \noi  as well as their restrictions to the split real forms of these complex semsimple groups. 
 
Higgs bundles over a compact Riemann surface $\Sigma$ of genus $g\geq 2$ were introduced  in \cite{N1} {and in a more general setting in \cite{simpson:1988}.}
 For any Lie group $G${,} a $G$-Higgs bundle on $\Sigma$ is a pair $(P,\Phi)$ where $P$ is a holomorphic principal bundle and $\Phi$ (the Higgs field) is a holomorphic section of an associated bundle twisted by $K$, the canonical bundle of the surface $\Sigma$.  If $G$ is a complex group then $P$ is a principal $G$-bundle, but if $G$ is a real form of a complex group then the structure group of $P$ is the complexification of a maximal compact subgroup of $G$.   {In this paper we consider only matrix groups, in particular $G=\SL(n,\C)$ or $\SO(2n,\C)$, and real forms of these groups. In these cases (described  in \S \ref{Fibration}) the Higgs bundles  {can be seen}  as holomorphic vector bundles with extra structure, where the precise nature of the extra structure is determined by the group $G$, and the Higgs fields are appropriately constrained sections of the endomorphism bundle twisted by $K$. }

 The defining data for  {$G$}-Higgs bundle serves to construct $G$-local systems on $\Sigma$. Indeed this is the crux of Non-Abelian Hodge Theory (NAHT), whereby the moduli space of semistable $G$-Higgs bundles on $\Sigma$ is identified with the moduli space of reductive representations of $\pi_1(\Sigma)$ into $G$. 
The implications of any group homomorphism $h:G_1\rightarrow G_2$ are clear for surface group representations, since composition with $h$ induces a map from representations into $G_1$ to representations into $G_2$.  It follows from NAHT that there must be a corresponding induced map between the Higgs bundle moduli spaces, but the transcendental nature of the NAHT correspondence means that the clarity of the induced map on representations does not transfer so easily to the induced map on Higgs bundles.  Our goal is to understand
 this map in the cases where the group homomorphism is given by the above isogenies.

In \cite{N2} Hitchin showed how the defining data of Higgs bundles can be re-encoded into a so-called spectral data set consisting of a ramified covering  $S$ of $\Sigma$  and a bundle on $S$. For the groups of interest in this paper, the spectral bundles are line bundles, and hence lie in Jacobians of the spectral curve. In fact, they must lie in Prym varieties determined by the $S$ and $\Sigma$. {These abelian varieties form the generic fibers in a fibration of the Higgs bundle moduli space over a half-dimensional linear space.}

{In terms of vector bundles, the effect of some low rank isogenies has been studied in \cite{aparicio}. In this paper we expand those result in the case of \eqref{isogenies3} through the use of Hodge star type operators. For both \eqref{isogenies} and  \eqref{isogenies3} we give a novel description of the spectral data correspondences, shedding new light on the maps between Higgs bundles and their moduli spaces.}
 Our key tool for understanding the induced maps on the spectral data for the Higgs bundles is a fibre-product construction, at the heart of which is the following diagram:
\begin{eqnarray}\label{diagram}
\xymatrix{
L_1\ar[d]&\CL=p_1^*L_1\otimes p_2^*L_2\ar[d]&L_2\ar[d]\\
S_1\ar[dr]_{\pi_1}&\ar[l]_{p_1}S_1\times_{\Sigma}S_2\ar[d]^{\pi}\ar[r]^{p_2}&S_2\ar[dl]^{\pi_2}\\
&\Sigma&
}
\end{eqnarray}

In our applications, the curves $S_i$ and line bundles $L_i$ come from spectral data  for $\SL(2,\C)$ or $\SL(4,\C)$-Higgs bundles, while the fiber-produc  $S_1\times_{\Sigma}S_2$, and the line bundle $\mathcal{L}$ yield the spectral data sets for the $\SO(4,\C)$ or $\SO(6,\C)$-Higgs bundles.  This construction clearly has wider applicability than the cases inspired by the isogenies \eqref{isogenies} and \eqref{isogenies3}.
One notable feature of the construction is that it generically yields smooth curves, even though curves defined by spectral data for $\SO(2n,\C)$-Higgs bundles are necessarily singular.  Indeed, the curves provided by our construction are the normalizations of the singular $\SO(2n,\C)$-spectral curves.
%
%
%

  {As seen in \cite{aparicio},  the isogeny $\CI_{2}$  innduces a map  {on $SL(2,\C)$-Higgs bundles $(E_i,\Phi_i)$ given by}  \begin{eqnarray}
\mathcal{I}_2((E_1,\Phi_1), (E_2,\Phi_2)) = \left[E_1\otimes E_2, \Phi_1\otimes 1+1\otimes\Phi_2\right].
 \end{eqnarray} 
  where $E_1\otimes E_2$ has orthogonal structure determined by the symplectic structures on $E_i$.  
Hence, in the case of the rank 2 isogeny $\CI_{2}$, on the generic fibers of the Hitchin fibration for the moduli space $\CM_{SL(2,\C)\times SL(2,\C)}$, our main result is the following (see Propositions \ref{I2spec}, \ref{I2realhiggs}, \ref{season2}):} 

\begin{theorem}\label{teorema2}   Let $S_i$ be the spectral curve of the $SL(2,\C)$-Higgs bundles $(E_i,\Phi_i)$, and \linebreak $L_i\in Prym(S_i,\Sigma)$ the corresponding spectral line bundle. Then the spectral data for the the image   $\mathcal{I}_2((E_1,\Phi_1), (E_2,\Phi_2))$ is given by $(\hat{S}_4,\mathcal{L})$ where 

\begin{itemize}
\item $\hat S_4:=S_1\times_{\Sigma}S_2$ is a smooth ramified fourfold cover, and
\item $\CL:= p_1^*(L_1)\otimes p_2^*(L_2)$ where $p_i:S_1\times_{\Sigma}S_2\rightarrow S_i$ are the projection maps.
\end{itemize}

The map is a $2^{2g}$ fold-coverings onto its images.  Restricted to $SL(2,\R)\times SL(2,\R)$-Higgs bundles, the image lies only in components of $\CM_{SO_0(2,2)}$ in which the Higgs bundles satisfy a topological constraint.
 \end{theorem}

 In the case of the rank 3 isogeny, as seen in \cite{aparicio}, the map $\CI_{3}$   induces a morphism  
 on $\SL(4,\C)$-Higgs bundles $(E,\Phi)$ given by
 \begin{eqnarray}\label{CL3}
\CI_3(E,\Phi) = \left[\Lambda^2 E, \Phi\otimes 1+1\otimes\Phi\right].
\end{eqnarray} 
where $\Lambda^2E$ has orthogonal structure determined by the isomorphism $\Lambda^2E^*\otimes\det(E)\simeq\Lambda^2E$.

This case differs from the previous one in two notable ways.  The first is seen in the restriction of \eqref{CL3} to $\SL(4,\R)$-Higgs bundles, where the map reflects the fact that $\SO(4)$ is not a simple Lie group.  This leads to a decomposition of the bundle $\Lambda^2E$  that is analogous to the decomposition of 2-forms on a Riemannian 4-manifold into self-dual and anti-self-dual forms. 

The other new wrinkle is in the induced map on spectral data where the fiber product in diagram \eqref{diagram} is now taken on two copies of the same covering, say $S$. The resulting curve can never be smooth since it has a diagonally embedded copy of $S$ which intersects other components at the ramification points for the covering $S\rightarrow \Sigma$. It turns out though that this component plays no part in the map induced by $\mathcal{I}_3$. We thus modify our construction by removing the diagonally embedded component in the fiber product of an $\SL(4,\C)$-spectral data set with itself  {(see Section \ref{isosmooth}, Proposition \ref {realI3}, \ref{rank3main}, \ref{propinv1} for details): }
   
   \begin{theorem} \label{teorema1}  Restricted to the Higgs bundles for split real form $\SL(4,\R)\subset\SL(4,\C)$, the induced map yields
\begin{eqnarray}
\CI_3(E,\Phi) = \left[\Lambda^2_+ E\oplus\Lambda^2_-E, \begin{pmatrix}0&\alpha\\-\alpha^{\rm{T}}&0\end{pmatrix}\right].
\nonumber
\end{eqnarray} 

\noi where $\Lambda^2_{\pm}E$ are the $\pm 1$-eigenbundles for an involution $*:\Lambda^2E\rightarrow\Lambda^2E$ determined by an orthogonal structure on $E$, with canonically determined orthogonal structures, and with $\Phi$ and  $\alpha$ related as in Eq.~\eqref{I*phi}.   
  {On the generic fibers of the Hitchin fibration for the moduli space $\CM_{SL(4,\C)}$, 
 given the spectral data $(S,L)$ corresponding to $(E,\Phi)$,} 
  the spectral data corresponding to the $\SO(6,\C)$-Higgs bundle $\CI_3(E,\Phi)$ is   $(\hat{S}_6,\CI_3(L))$, where
\begin{itemize}
\item $\hat{S}_6$ is the symmetrization of the non-diagonal component in the fibre product $S\times_{\Sigma} S$; 
\item $\CI_3(L)$ is a canonical twist of the line bundle generated by 
local sections of $\mathcal{L}=p_1^*(L)\otimes p_2^*(L)$ that are anti-invariant with respect to the symmetry of $S\times_{\Sigma} S$.
\end{itemize}

The map is a $2^{2g}$ fold-coverings onto its images.  Restricted to points representing $SL(4,\R)$-Higgs bundles, the image lies only in components of $\CM_{SO_0(3,3)}$ in which the Higgs bundles satisfy a topological constraint.
\end{theorem}

While the two cases covered by \eqref{isogenies} and \eqref{isogenies3} are by no means the only interesting ones, they have some unique features and   serve to illustrate phenomena that we expect to apply in greater generality.  The group $\SL(2,\R)$ has a distinguished place in any discussion of surface group representations because of its relation to hyperbolic structures and Teichmuller space, while the group $\SO_0(2,2)$ is the isomtery group of the anti-de Sitter space $AdS^3$.   {Moreover, b}oth groups are  split real forms of complex semisimple groups (viz. $\SL(2,\C)$ and $\SO(4,\C)$ respectively) and also groups of Hermitian type. The only other such group is $\mathrm{Sp}(4,\R)$. For both classes of real Lie group, the representation variety or, equivalently, the moduli space of Higgs bundles has a distinguished set of components. In the case of the split real forms, the distinguished components are called ``higher Teichmuller components'' because they generalize the copies of Teichmuller space which occur in the case of $\SL(2,\R)$.  For the groups of Hermitian type, the $G$-Higgs bundles (or the surface group representations) carry a discrete invariant known as a Toledo invariant. This invariant satisfies a so-called Milnor-Wood bound and the distinguished components are those in which the invariant has maximal value.

 {In the two cases considered in this article,}
  the induced maps on the spectral curves yield a pairing between two different ramified coverings of a common base curve, together with isogenies between the associated Prym varieties. The spectral curves in each pair are, moreover, determined by different representations of the same group. This kind of situation, with pairings between coverings of a given curve and isogenies between the associated Prym varieties, has been considered in the context of integrable systems 
  and also   by Donagi (see e.g. \cite{ron, ron2}).   As noted in \cite[Example 3]{ron}  the correspondence we see for the pair $\SL(4,\C)$ and $\SO(6,\C)$, namely between a four-fold covering and a six-fold covering with a fixed-point-free involution, is essentially the correspondence described by Recillas in his trigonal construction \cite{recillas}  and generalized by Donagi in \cite{ron2}.   The novelty in our version of the correspondence --  and resulting relation between Prym varieties -- lies in the use of fiber products to get explicit descriptions of the maps. 
 
The first part of this paper covers background material on Higgs bundles (Section \ref{Fibration}), spectral curves (Section \ref{spectral_higgs}) and the isogenies (Section \ref{Lie}).  In Section \ref{isosingular} we describe the maps induced by the isogeny $\CI_2$ on both the complex group $\SL(2,\C)\times\SL(2,\C)$ and its split real form, and in Section \ref{isosmooth} we do the same for $\CI_3$. We include a discussion of the relation between our construction on spectral data and the trigonal construction of Recillas, and show how the map we obtain between Prym varieties can be interpreted in terms of a correspondence between curves. We conclude, in Section \ref{maps}, with a discussion of maps between moduli spaces.
 
\subsection*{Acknowledgements}
We thank David Baraglia, Ron Donagi, Tamas Hausel, Nigel Hitchin, Sheldon Katz, Herbert Lange, Tom Nevins, Tony Pantev, Mihnea Popa, S. Ramanan, Hal Schenck,  and Michael Thaddeus,  for helpful and enlightening discussions.  Both authors acknowledge support from U.S. National Science Foundation grants DMS 1107452, 1107263, 1107367 ``RNMS: GEometric structures And Representation varieties (the GEAR Network)."


 
\section{Higgs bundles and the Hitchin fibration}\label{Fibration}

Let  $\Sigma$ be a compact Riemann surface of genus $g\geq 2$,  {and}   {$\pi:K:=T^*\Sigma\rightarrow \Sigma$} its canonical bundle. For $G_c$ a complex reductive Lie group with Lie algebra $\mathfrak{g}_c$, from \cite{N1} one has the following definition:

\begin{definition}\label{Gc-Higgs} A $G_c$-Higgs bundle on $\Sigma$ is given by a pair $(P,\Phi)$ for $P$ a principal $G_c$-bundle on $\Sigma$, and $\Phi$ a holomorphic section of ${\rm Ad}P\otimes K$, for ${\rm Ad}P=P\times_{Ad}\mathfrak{g}_c$  the adjoint bundle associated to $P$. \label{def complex}
\end{definition}
 For  $GL(n,\C)$ one recovers classical Higgs bundles as introduced in \cite{N2}. For matrix groups, the definition can be reformulated in terms of vector bundles rather than principal bundles.  In particular{,}
 \begin{itemize}
\item  {An} $SL(n,\C)$-Higgs bundle is a pair $(E,\Phi)$  where  $E$ is a rank $n$ holomorphic bundle on $\Sigma$ with fixed trivial determinant, and $\Phi$   a traceless holomorphic section of $\mathrm{End}(E)\otimes K${;} 
\item  {An} $\SO(n,\C)$-Higgs bundle   is a pair $(E,\Phi)$  where $E$ is a rank $n$ holomorphic bundle on $\Sigma$ with an orthogonal structure $Q$ and a compatible trivialization of its determinant bundle\footnotemark\footnotetext{A trivialization $\delta:\det(E)\simeq\mathcal{O}_{\Sigma}$ is compatible with $Q$ if $\delta^2$ agrees with the trivialization of $(\Lambda^nE)^2$ given by the discriminant $Q:\Lambda^nE\rightarrow\Lambda^nE^*$ (see Remark 2.6 in \cite{ramanan81})}, and $\Phi$ is a holomorphic section of $\mathrm{End}(E)\otimes K$ satisfying  {$Q(u,\Phi v)=-Q(\Phi u, w)$}.
\end{itemize}

 
The construction of $G$-Higgs bundles for a real form $G$ of $G_c$ goes back to \cite{N1, N5} (see also \cite{Go2, bradlow-garcia-prada-gothen:2005}, for further details). The definition requires a choice of a maximal compact subgroup of $H\subset G$,  and the Cartan decomposition  
$\mathfrak{g}=\mathfrak{h}\oplus \mathfrak{m},$
for $\mathfrak{h}$ the Lie algebra of $H$, and $\mathfrak{m}$ its orthogonal complement. 
Note that the Lie algebras satisfy the symmetric space relations
$ [\mathfrak{h}, \mathfrak{h}]\subset\mathfrak{h}$,  $[\mathfrak{h,\mathfrak{m}}]\subset\mathfrak{m}$, and $[\mathfrak{m},\mathfrak{m}]\subset \mathfrak{h}\nonumber
$. Hence there is an induced isotropy representation given by
$\text{Ad}|_{H^{\mathbb{C}}}: H^{\mathbb{C}}\rightarrow GL(\mathfrak{m}^{\mathbb{C}}). $

\begin{definition}\label{real form}
 Given $G$ a real form of a complex Lie group $G_c$, a \text{principal} $G${\em-Higgs bundle} is a pair $(P,\Phi)$ where
  $P$ is a holomorphic principal $H^{\mathbb{C}}$-bundle on $\Sigma$, and 
  $\Phi$ is a holomorphic section of $P\times_{Ad}\mathfrak{m}^{\mathbb{C}}\otimes K$.
\label{defrealhiggs}
\end{definition}
%
%
Alternatively,  {as done in} \cite{N5}, we may regard real Higgs bundles as classical Higgs bundles $(E,\Phi)$,   with extra conditions  reflecting the structure of the real group and its isotropy representation.   In this paper we shall mainly consider 
 $\SL(n,\R)$ and $\SO(n,n)$-Higgs bundles, for which we recall their main properties in Sections \ref{spectral_slr} and \ref{real so} (for further references, see  \cite{aparicio, N5, thesis}).

Following \cite{N1}, a classical Higgs bundle $(E,\Phi)$ is said to be (semi)-stable if  all subbundles $F\subset E$ such that $\Phi(F)\subset F\otimes K$ satisfy $\deg(F)/{\rm rk}(F)< (\leq)\deg(E)/{\rm rk}(E)$. Moreover, the pair is said to be poly-stable if it can be written as a direct sum of stable Higgs bundles $(E_i,\Phi_i)$ for which    $\deg(E_i)/{\rm rk}(E_i)=\deg(E)/{\rm rk}(E)$. The notion of stability can be extended to $G_c$-Higgs bundles as in \cite{N2}, as well as to $G$-Higgs bundles (e.g. see \cite{steve}), and used to construct the corresponding moduli spaces.  As explained below in Remark \ref{allstable} stability considerations will not play a role in this paper.
 We shall denote by $\CM_{G_c}$ and $\CM_{G}$ the moduli space of $G_c$-Higgs bundles and $\CM_{G}$-Higgs bundles, respectively.

%


\subsection{The Hitchin fibration and spectral curves}\label{fibration_hitchin}
Given a homogenous basis $r_1,\ldots,r_k$ of invariant polynomials for the Lie algebra of a complex Lie group $G_c$, let $d_1,\ldots,d_k$ be their degrees. Then, the Hitchin fibration (introduced in \cite{N2}) is defined as  
\begin{eqnarray}\nonumber
h:\mathcal{M}_{G_c}&\rightarrow& \mathcal{A}_{G_c}:=\bigoplus_{i=1}^{k}H^{0}(\Sigma, K^{d_i}) \\
(E,\Phi)&\mapsto& (r_1(\Phi), \ldots, r_k(\Phi)).\nonumber
\end{eqnarray}

A point in the Hitchin base $\mathcal{A}_{G_c}$ {determines a section of the line bundle $\pi^*K^{d_k}$ defined by   {$\eta^{d_k}+\dots+r_{k-1}(\Phi)\eta^{d_1} +r_k(\Phi)$}, where $\eta$ the tautological section of $\pi^*K$.  The zero locus of this section is the so-called spectral curve  {$S$} associated to the Higgs bundle.   {The curve $S$} lies the total space of $K$ and the projection $\pi:S\rightarrow \Sigma$  is a $d_k$-fold cover of $\Sigma$.}  By considering vector bundles on it (or its desingularization), one can give a geometric description  {of} the fibres of the Hitchin fibration \cite{N2}.  In the case of classical Higgs bundles the generic fibres  are given by the Jacobian varieties of the spectral curves  $S${,}  on which $\Phi$ has a single-valued eigenvalue $\eta$.  {Indeed, given a line bundle $L$ on $S$ one may recover the corresponding classical Higgs bundle $(E,\Phi)$ by taking $E:=\pi_*L$  and $\Phi$   the direct image of $ \eta: L \rightarrow L\otimes  \pi^*K$.}
  In Section \ref{spectral_higgs} we describe more fully the spectral data for the specific cases of interest in this paper.

By considering the  embedding of $\mathcal{M}_G$ in  {$\mathcal{M}_{G_c}$}, one may identify the $G$-Higgs bundles as points in the Hitchin fibration satisfying additional constraints (see \cite{N2} for classical complex Lie groups, and  \cite{singapore} and references therein for real Higgs bundles). In particular, in the case of the split real form of $G_c$, the line bundles are the torsion two points in the Jacobian (see \cite[Theorem 4.12]{thesis}).

\begin{remark}
Spectral curves are defined for all Higgs bundles represented by points in the moduli spaces $\mathcal{M}_G$, but they are not necessarily smooth. It follows however from Bertini's theorem that on the generic fiber of the Hitchin fibration the spectral curves are smooth. Throughout the paper we shall consider Higgs bundles over the smooth loci of the Hitchin fibration, i.e. points defining smooth spectral curves, since we will further  {restrict our attention to those curves for which} the most generic type of ramification behaviour occurs. 
\end{remark}

\begin{remark}\label{allstable}
 {Note that} if the spectral curve for a Higgs bundle is smooth then the characteristic polynomial for the Higgs field is irreducible. The Higgs bundle thus has no invariant subbundles and is therefore automatically stable. Hence, whilst stability is needed to define the moduli spaces,  it will play no role in our discussions and we shall omit any further comment on it. 
\end{remark}

\section{Spectral data for complex and real Higgs bundles}\label{spectral_higgs}
We shall recall here how to study the fibres of the Hitchin fibration through spectral data, which we shall do by reviewing the methods introduced in \cite{N2, N5} for complex Higgs bundles, and in \cite{thesis, singapore} for real Higgs bundles. Since our main focus is on $SO(2n,\C)$ and $SL(2n,\C)$-Higgs bundles, we shall restrict attention to those groups and their split real forms. 

\subsection{Spectral data for $SL(n,\C)$-Higgs bundles}\label{spectral_sl} 
  The spectral curve $\pi: S\rightarrow \Sigma$ associated to  {an $SL(n,\mathbb{C})$-Higgs bundle   $(E,\Phi)$} has equation \begin{eqnarray}\label{spectralsl}
\text{det}(I \eta-\Phi)= \eta^{n}+a_{2}\eta^{n-2}+\ldots + a_{n-1}\eta+a_{n}=0\ .
\end{eqnarray}
 Here $a_{i}\in H^{0}(\Sigma,K^{i})$, and  $\eta$ is the tautological section of $\pi^*K$.  The $(n-1)$-tuple $(a_2,\dots,a_n)$ defines a point in the base of the Hitchin fibration  {(see \cite{N2})}. Over a dense open set in the base of the Hitchin fibration, the spectral curve $S$ is smooth and the fiber can be identified  {(biholomorphically)} with  {the Prym variety  ${\rm Prym}(S,\Sigma)$,}  {the}  subset of $\text{Jac}(S)$  of line bundles  whose direct image sheaf has trivial determinant. The relation between $L\in  \mathrm{Prym}(S,\Sigma)$ and $E$  {(once a choice of $K^{1/2}$ is made)} is then

\begin{equation}\label{EfromL}
E:=\pi_*(L\otimes\pi^*(K^{(n-1)/2})).
\end{equation}
  {Moreover,  
$\mathrm{Prym}(S,\Sigma)=\{L\in\mathrm{Jac}(S)\ |\ Nm(L)=\mathcal{O}_{\Sigma}\}
$,  where $Nm$ denotes the norm map defined by the projection $\pi:S\rightarrow\Sigma$,}
 the trivialization of $\det(E)$ determined by the trivialization of $Nm(L)$.   Note that in the case of two-fold covers,  a line bundle is in the Prym variety if and only if its dual is isomorphic to the pull back by the involution switching the sheets of  the cover.

\subsection{Spectral data for $SL(n,\R)$-Higgs bundles}  \label{spectral_slr}

From Definition \ref{real form} {,} Higgs bundles with structure group $SL(n,\R)$ are given by classical Higgs bundles $(E,\Phi)$ together with an oriented orthogonal structure on $E$, with respect to which the Higgs field is traceless and symmetric. 
Moreover, by  \cite[Theorem 4.12]{thesis} one has that the
  intersection of the moduli space $\CM_{SL(n,\R)}$    with  the smooth fibres of the $SL(n,\C)$ Hitchin fibration is given by line bundles $L\in {\rm Prym}(S,\Sigma)$ such that $L^2\cong \mathcal{O}$.  
Following \cite{BNR} and \cite{N2}, a torsion 2 line bundle $L$ induces an $SL(n,\R)$-Higgs bundle $(E,\Phi)$ with $E=\pi_*(L\otimes \pi^*K^{(n-1)/2}) $ and $\Phi$ the push down of the tautological section $\eta$. Moreover, the orthogonal structure on $E$ comes from an $O(1)$ structure on $L${.}
 

\subsection{Spectral data for $SO(2n,\C)$-Higgs bundles} \label{season}   From the characteristic polynomial of  {an $SO(2n,\C)$-Higgs field} $\Phi$ one obtains  a $2n$-fold cover $\pi:S\rightarrow \Sigma$ whose equation is given by
\begin{eqnarray}\text{det}(\eta I-\Phi)=\eta^{2n}+b_{1}\eta^{2n-2}+\ldots+b_{n-1}\eta^{2}+p_{n}^{2}=0,\label{curvesonn}\end{eqnarray}
   {where $\eta$ is as  in \eqref{spectralsl}, the sections} $b_{i}\in H^{0}(\Sigma,K^{2i})$,  and   $p_{n}\in H^{0}(\Sigma, K^{n})$ is the Pfaffian of $\Phi$. This curve  {has a natural involution $\sigma:\eta\mapsto -\eta$, and is singular}  at   $\eta=p_n=0$, i.e., at the fixed points of the  {$\sigma$} . 
The normalization of $S$,  {which we   denote by}
 $\hat{\pi}:\hat{S}\rightarrow \Sigma$, is what we shall refer to as the spectral curve.  {As in the previous cases,} $\hat{S}$  {is generically}  smooth,  {and t}he 
  involution $\sigma$ extends to an involution $\hat{\sigma}$ on $\hat{S}$ which does not have fixed points.  
 
 The generic fibers of the Hitchin fibration can be identified with an abelian variety defined by a connected component of  {$\Prym(\hat{S},\hat{S}/{\sigma})$, i.e.} the kernel of the norm map  {$Nm:Pic(\hat{S})\rightarrow Pic(\hat{S}/{\sigma})$}. Given a line bundle $L\in \Prym(\hat{S},\hat{S}/{\sigma})$  {and a choice of $K^{1/2}$},  the vector bundle $E$ is recovered as 
%
\begin{equation}\label{EfromL-SO}
E:=\pi_*(L\otimes (K_{\hat{S}}\otimes\pi^*(K^*))^{1/2}).
\end{equation}

 {In this case, t}he orthogonal structure on $E$ comes from the isomorphism $\sigma^*{L}\simeq L^{-1}$ (by virtue of which $Nm(L)=0$){.} 
  {T}he spectral data associated to an $SO(2n,\mathbb{C})$-Higgs bundle is defined on the desingularization $\hat{S}$ of $S$. From Eq.~ \eqref{curvesonn}, each pair of points $(b_1,\ldots, b_{n-1}, p_n)$ and $(b_1,\ldots, b_{n-1}, -p_n)$ define {s} the same curve $S$ and desingularization $\hat{S}$.  The Prym variety  $\text{Prym}(\hat{S},\hat{S}/\hat{\sigma})$  has two connected components.   {Moreover,  f}rom   \cite[Lemma 1]{mumford}  all line bundles $L$ on $\hat S$ which satisfy $Nm(L)\simeq \mathcal{O}$ are of the form
$L=N\otimes\hat \sigma^*(N^*),$
 for some line bundle $N$ on $\hat S$ of degree  $0$ or $1$. The two connected components of $\text{Prym}(\hat{S},\hat{S}/\hat{\sigma})$  correspond to the two possibilities for the parity of the degree   $\deg(N)$.  This can also be seen as a reflection of  the fact that $\mathcal{M}_{\SO(2n,\C)}$ has two components corresponding to the possible values for the second Stieffel-Whitney class of an $\SO(2n,\C)$-bundle.
We say that the spectral data is given by $(\hat S, L)$ for $\hat S$ the normalized curve and  $L\in \text{Prym}(\hat{S},\hat{S}/\hat{\sigma})$.

 \begin{remark}\label{ramdiv} For any $m$-fold ramified cover $\pi: S\rightarrow\Sigma$,  the ramification divisor $R$ in $S$ relates the canonical bundles $K_S$ and $K$  of $S$ and $\Sigma$, respectively, by
%
$[R]=K_S\otimes\pi^*K^*,$
where $[R]$ denotes the line bundle defined by the divisor $R$. If $S$ is a spectral curve, i.e. contained in the total space of $K$, then $
K_S= \pi^*K^{m}
$, and $[R]=\pi^*K^{(m-1)}$.  The relation between $L$ and $E$ in both \eqref{EfromL} and\eqref{EfromL-SO} can thus be given\footnotemark
 \footnotetext{Note that the line bundle corresponding to our $L$ is denoted in \cite{N2} by $U$, so that the line bundle denoted in \cite{N2} by $L$ corresponds to $L[R]^{-1/2}$ in our notation.}
 as 
\begin{equation*}
E=\pi_*(L\otimes [R]^{1/2})\ .
\end{equation*}

\end{remark}


\subsection{Spectral data for $\SO_0(n,n)$-Higgs bundles}\label{real so} 
From Definition \ref{real form}  one has that $SO(n,n)$-Higgs bundles can be viewed as $\SO(2n,\C)$-Higgs bundles of the form  $(W_1\oplus W_2, \Phi)$ where $W_i$ are vector bundles of rank $n$ with orthogonal structure, say $q_i$, and \begin{small}
\begin{equation}\Phi=\left(\begin{array}{cc}
              0&\alpha\\
-\alpha^{\rm T}&0
             \end{array}
\right),
\label{higgsalpha}\end{equation} \end{small}
for $\alpha^{\rm T}$ the orthogonal transpose of $\alpha$, i.e. $\alpha^T=q_2^{-1}\cdot\alpha^*\cdot q_1$ where $\alpha^*$ denotes the dual map. By further requiring  that $\det(W_i)\simeq\mathcal{O}$, one obtains an $SO_0(n,n)$ pair.  In this case, from \cite[Theroem 4.12]{thesis} and along the lines of Section \ref{spectral_slr} one has that the intersection of the moduli space $\CM_{SO_0(n,n)}$    with  the smooth fibres of the $\SO(n,\C)$ Hitchin fibration is given by $L\in {\rm Prym}(\hat{S},\hat{S}/\hat{\sigma})$ which satisfy $L^2\cong \mathcal{O}$.
As for classical Higgs bundles, given a torsion two line bundle $L\in {\rm Prym}(\hat{S},\hat{S}/\hat{\sigma})$ one obtains an $SO_0(n,n)$-Higgs bundle $(E,\Phi)$ by taking the direct image of the line bundle $L \otimes (K_{\hat S}\otimes \pi^* K^*)^{1/2}$  and the push forward of  $\eta$.

 Given the spectral data $(\hat{S},L)$ of an   $\SO_0(n,n)$-Higgs bundle, since $L$ is  {of order two and}  in the Prym variety  {of a two fold cover $p:\hat{S}\rightarrow\hat{S}/{\sigma}$,}    it is invariant under the involution $\sigma$ on $\hat{S}$.  Hence its local sections decompose into invariant and anti-invariant local sections {,}       {and thus}  the direct image   decomposes as $p_*L=\hat{L}_1\oplus\hat{L}_2$,   {where the summands are}  generated by local  {invariant and} anti-invariant sections (see \cite{yoU}). 
Moreover, considering the $n$-fold cover 
$\hat \pi: \hat{S}/ \hat \sigma\rightarrow \Sigma, $
 the orthogonal bundles $W_i$ are recovered by taking  $\hat \pi_*(L_i\otimes p_*(K_{\hat S}\otimes \hat \pi^* K^*)^{1/2})$  {for $i=1,2$}.

\section{Homomorphisms of groups and induced maps} \label{Lie}

 As noted in the Introduction, given a fixed surface $\Sigma$ and a  homomorphism between two Lie groups $\Psi:G\rightarrow G'$, there is clearly an induced map 
$$\Psi:\mathrm{Rep}(\Sigma,G)\rightarrow \mathrm{Rep}(\Sigma,G'),$$
\noi where $\mathrm{Rep}(\Sigma,G)$ denotes the space of representations modulo conjugation. The correspondence between surface group representations and Higgs bundles thus implies a similar induced map between $G$-Higgs bundles and $G'$-Higgs bundles.  

From Definition \ref{Gc-Higgs}  one sees that for an isomorphism between complex groups there is in fact an induced map from $G$-Higgs bundles to $G'$-Higgs bundles given by 
\begin{equation}\label{map-principalversionC}
\Psi_*: (P_G,\Phi)\mapsto (P_{G'}=P_G\times_{\Psi}G',\Phi'=d\Psi(\Phi)),
\end{equation}
\noi where 
\begin{equation}\label{dPsi}
d\Psi: ad(P_G)\rightarrow ad(P_{G'})
\end{equation}
\noi  is the map defined by the derivative at the identity of the map $\Psi$, i.e. by the map on Lie algebras.  {Moreover, if the homomorphism $\Psi$ restricts to a map between real forms of $G$ and $G'$ which respects the Cartan decompositions of the Lie algebras,}  then the map $\Psi$ induces a map from $G_r$-Higgs bundles to $G_r'$-Higgs bundles.  From Definition \ref {real form},  {following the notation of Section \ref{Fibration}}, the map is given by
\begin{equation}\label{map-principalversionR}
\Psi_*: (P_{H_{\C}},\Phi)\mapsto (P_{H'_{\C}}=P_{H_{\C}}\times_{\Psi}H_{\C}',\Phi'=d\Psi(\Phi).
\end{equation}
%
%


 The maps in \eqref{map-principalversionC} and \eqref{map-principalversionR} ought to acquire descriptions purely in terms of spectral data {, and this shall be investigated in forthcoming sections}.  The relation between the spectral curves can be deduced from its relation to the eigenvalues of the Higgs field. Thus if the spectral curve $S$ for $(P_{H_{\C}},\Phi)$  is defined by 
\begin{equation}
0=\det(\eta I-\Phi)=\prod_{i=1}^n(\eta-\eta_i),
\end{equation}

\noi where $\{\eta_1,\dots,\eta_n\}$ are the eigenvalues of $\Phi${,} then the spectral curve $S'$ for $\Psi_*(P_{H_{\C}},\Phi)$   {is}  defined by an equation of the same form except with the eigenvalues replaced by the eigenvalues of $d\Psi(\Phi)$, leading to a map $S\mapsto S'$. 

Alternatively  one may seek a more intrinsic understanding of the map purely in terms of the information encoded in the geometry of the covering $S\rightarrow \Sigma$ and the restrictions on the spectral line bundle $L$. This is our goal for the special cases of the isogenies in \eqref{isogenies} and \eqref{isogenies3}.
The isogenies can be described in several ways, including from coincidences of  Dynkin diagrams or  in terms of representations e.g.  in terms of Schur functors as in   \cite[Chapter 9-10]{claudio}. For our purposes,  the representation theoretic point of view is convenient, as described in the next two sections.
 
 \subsection{The isogeny between $\SL(2,\C) \times \SL(2,\C)$ and $ \SO(4,\C)$}\label{Lie2}
 
 The group $SL(2,\C) \times SL(2,\C) $ acts on $\C^2 \otimes \C^2=\C^4$ by 
$(g,h) (v \otimes w) := (gv) \otimes (hw)$. The isogeny onto $\SO(4,\C)$ can be seen as coming from the fact that $\SL(2,\C)\simeq\mathrm{Sp}(2,\C)$. If $\omega$ is the symplectic form on $\C^2$  preserved by matrices with unit determinant, then $Q_4:= \omega \otimes \omega$, defines a symmetric, non-degenerate bilinear form on $\C^2 \otimes \C^2$. Hence one has a map   
\begin{eqnarray}\label{abovemap}
 \CI_2:\SL(2,\C) \times \SL(2,\C)&\rightarrow &\SO(4,\C), \\
 (A_1,A_2)&\mapsto &A_1\otimes A_2,\nonumber
\end{eqnarray} 
\noi where $\SO(4,\C)$ is the group  of (orientation preserving) linear maps ${\C}^{4} \rightarrow {\C}^{4}$ preserving the  form $Q_4$. The derivative at the identity yields an isomorphism of Lie algebras given by
\begin{eqnarray}\label{I2algebra}
 d{\CI_2}:\mathfrak{sl}(2,\C) \times \mathfrak{sl}(2,\C)&\rightarrow& \mathfrak{so}(4,\C), \\
 (\dot{A_1},\dot{A_2})&\mapsto & \dot{A_1}\otimes I+ I\otimes \dot{A_2}.\nonumber
\end{eqnarray} 

\begin{remark}\label{evalsadd}
If $\dot{A_i}$ has eigenvalues $\{\lambda^i_1,\lambda^i_2\}$, for $i=1,2$, then the image $ d{\CI_2} (\dot{A_1},\dot{A_2})$ has eigenvalues $\{\lambda^1_a+\lambda^2_b\ |1\le a,b\le 2\}$.  In particular, if $Tr(\dot{A}_i)=0$ then $\lambda^i_2=-\lambda^i_1$ and the eigenvalues for $ d{\CI_2} (\dot{A_1},\dot{A_2})$ are $\{\pm\lambda^1_1\pm\lambda^2_1\}$.
\end{remark}

Restricted to $\R^2\otimes\R^2=\R^4$, the quadratic form $Q_4$ has signature $(2,2)$ and thus the above map  {\eqref{abovemap}} between complex Lie groups restricts to  
\begin{eqnarray}\CI_2:SL(2,\R)\times SL(2,\R)\rightarrow SO_0(2,2),\label{mapreal22}\end{eqnarray}
  where the subscript in $\SO_0(2,2)$ denotes the connected component of the identity (see  \cite[Section 5.2]{aparicio} for more details). The map on Lie algebras similarly restricts.  {Indeed, given}  $\dot{a}_i\in\mathfrak{sl}(2,\R)$ 
   symmetric and trace-free  {for $i=1,2$}, and  {fixing}  a basis for $\R^4$ such that the orthogonal structure has the form ${\tiny\begin{pmatrix}I&0\\0&-I\end{pmatrix}}$, then the map in \eqref{I2algebra}  {is given by} 
\begin{equation}\label{I2algebra2}
d{\CI_2}(\dot{a}_1,\dot{a}_2)=  \begin{bmatrix}0&\alpha\\\alpha^t&0\end{bmatrix}\in \mathfrak{so}(2,2),
\end{equation}

\noi where $\alpha^t$ denotes the transpose\footnote{Notice that $\alpha^t=-(-I\cdot\alpha^t\cdot I)$ so that in this case $\alpha^t=-\alpha^T$.}. The precise form of $\alpha$ depends on the orientation chosen for $\R^4$, i.e. on the identification $\Lambda^4\R^4\simeq\R$.   {Moreover, whilst }  $\mathrm{Pf}(d\CI_2(\dot{a}_1,\dot{a}_2))^2=\det(\alpha)^2${,}  the sign of $\mathrm{Pf}(d\CI_2(\dot{a}_1,\dot{a}_2))$ depends on the choice of orientation.

%

\subsection{The isogeny between $\SL(4,\C)$ and $\SO(6,\C)$}\label{Lie1}

 The group $SL (4, \C)$ of volume-preserving linear maps $T: {\C}^{4} \rightarrow {\C}^{4}$
 has a   6-dimensional representation on the exterior power $\Lambda^2 \C^4$ given by
$\displaystyle  g( v \wedge w ) := (gv) \wedge (gw)$.  Taking  $n=4$ and $k=2$ in the isomorphism
\begin{equation}\label{QforLambda2}
\Lambda^k(\C^*)\otimes\Lambda^n(\C)\longrightarrow \Lambda^{n-k}(\C),
\end{equation}
 and fixing an identification $\Lambda^4\C^4\simeq\C$ (i.e. fixing a volume form), one gets a bilinear form $Q_6$ on $\Lambda^2 \C^2$ which is symmetric and non-degenerate. Since an element of $SL(4,\C)$ preserves the volume form, it preserves  {a} bilinear form, and thus one has a map 
\begin{eqnarray}
 \CI_3:\SL(4,\C)&\rightarrow &\SO(6,\C) \\
 A&\mapsto &\Lambda^2A,\nonumber
\end{eqnarray}

\noi  where $SO(6,\C)$ is the group  of (orientation preserving) linear maps ${\C}^{6} \rightarrow {\C}^{6}$ preserving the non-degenerate symmetric form $Q_6$.   {Moreover,} the map gives a double cover $SL(4,\C)$ of $SO(6,C)$, thus realising $SL(4,\C)$ as $Spin(6,\C)$.  
\noi The derivative at the identity gives the Lie algebra isomorphism 
\begin{eqnarray}\label{I3algebra}
 d{\CI_3}:\mathfrak{sl}(4,\C)&\rightarrow& \mathfrak{so}(6,\C), \\
 \dot{A}&\mapsto & \dot{A}\otimes I_4+ I_4\otimes \dot{A},\nonumber
\end{eqnarray} 

\noi where $I_4$ denotes the identity map on $\C^4$ and the endomorphism $\dot{A}\otimes I_4+ I_4\otimes \dot{A}$ is understood to be the restriction to $\Lambda^2\C^4\subset\C^4\otimes\C^4$.

\begin{remark}\label{eigensum}
If $\dot{A}$ has eigenvalues $\{\lambda_a\}_{a=1}^{4}$, then as a map on $\C^6=\Lambda^2\C^4$ the image $ d{\CI_3} (\dot{A})$ has eigenvalues  {$\{\lambda_a+\lambda_b\ |1\le a<b\le 4\}$.} \end{remark}

Restricted to $\R^6$, the quadratic form $Q_6$ has signature $(3,3)$ and thus $\CI_3$ restricts to 
\begin{eqnarray}\label{I3group}\CI_3:SL(4,\R)\rightarrow SO_0(3,3),\label{mapreal33}\end{eqnarray}
  with a corresponding restriction of \eqref{I3algebra}. In particular, if $\dot{a}\in\mathfrak{sl}(4,\R)$ is symmetric, trace-free, with entries $a_{ij}$,  {fixing} a basis for $\R^6$ such that the orthogonal structure has the form ${\tiny\begin{pmatrix}I&0\\0&-I\end{pmatrix}}$,  the map has the form
\begin{equation}\label{I3algebrasym}
d\CI_3(\dot{a})=  \begin{bmatrix}0&\alpha\\\alpha^t&0\end{bmatrix}\in \mathfrak{so}(3,3).
\end{equation}
  {As in the previous case, t}he precise form of $\alpha$ depends on the orientation chosen for $\R^6$, i.e. on the identification $\Lambda^6\R^6\simeq\R${, and a} standard choice yields
\begin{equation}\label{alphamatrix}
\alpha=\begin{bmatrix}
a_{13}+a_{24}&-a_{14}+a_{23}&a_{11}+a_{22}\\
-a_{12}+a_{34}&a_{11}+a_{33}&a_{14}+a_{23}\\
-a_{22}-a_{33}&a_{12}+a_{34}&-a_{13}+a_{24}
\end{bmatrix}.
\end{equation}

\noi In all cases, the Pfaffian is given by $Pf(d\CI_3(a))=\pm\det(\alpha)$, where the sign depends on the choice of orientation.

\begin{remark}\label{notsimple}
 {Fixing}  a maximal compact subgroup $\SO(4)\subset \SL(4,\R)${, one obtains} extra structure related to the presence of two inequivalent normal $\SO(3)$ subgroups, by virtue of which $\SO(4)$ fails to be simple.   These subgroups have an important influence on the map induced by $\CI_3$ on $\SL(4,\R)$-Higgs bundles {,}  described in more detail in \S \ref{Higgs33}, and  {considered} also   in \S \ref{inv33}.
\end{remark}


 \section{The rank 2 isogeny and   the Hitchin fibration}\label{isosingular}

In this section we describe the map {corresponding to the isogeny $\mathcal{I}_2$ in \S\ref{Lie2} in terms of spectral data}.   After exploring the maps for the complex groups we examine the extra conditions required to understand the corresponding maps for the split real forms.  
  
\subsection{ $SL(2,\C)\times SL(2,\C)$- and $SL(2,\R)\times SL(2,\R)$-Higgs bundles} \label{specsl2}

An $SL(2,\C)\times SL(2,\C)$-Higgs bundle on $\Sigma$ is defined by a pair of $\SL(2,\C)$-Higgs bundles  $(E_i,\Phi_i)$, for $i=1,2$.  An $SL(2,\R)\times SL(2,\R)$-Higgs bundle can be viewed as a pair of this type in which the bundles have oriented orthogonal structures and the Higgs fields are traceless and symmetric with respect to the orthogonal structures.  Up to isomorphism, we may thus assume in this case that the bundles are of the form $E_i=N_i\oplus N_i^{*}$ where $N_i$ is a line bundle of non-negative degree, the orthogonal structure is defined by the isomorphism 
\begin{equation}\label{qstd}
q=\begin{bmatrix}0&1\\1&0\end{bmatrix}:N_i\oplus N_i^{*}\rightarrow N_i^{*}\oplus N_i,
\end{equation}
 and the Higgs field is of the form 
\begin{equation}\label{higgs22}
\Phi_i=\begin{pmatrix}0&\beta_i\\\gamma_i&0\end{pmatrix}.
\end{equation}
 The characteristic polynomials for the Higgs fields define two spectral curves $\pi_i:S_i\rightarrow \Sigma$ in the total space of $K$ as in Eq.~\eqref{spectralsl}. These are 2-fold covers of the Riemann surface, with equations
$$\eta^2+a_i=0, $$
\noi for $a_i\in H^0(\Sigma, K^2)$ and $\eta$ the tautological section of $\pi^*K$. Generically the quadratic differentials $a_i$ have simple zeros, and thus by Bertini's theorem the curves $S_i$ are generically smooth. As seen in Section \ref{spectral_slr}, the spectral data associated to these $\SL(2,\C)$-Higgs bundles is completed by  line bundles $L_i\in {\rm Prym}(S_i,\Sigma)$.

In the case of $SL(2,\R)\times SL(2,\R)$-Higgs bundles,  by  \cite[Theorem 4.12]{thesis}  the line bundles are required to satisfy $L_i^2\cong \mathcal{O}_{S_i}$.    Note that, since $L_i\in {\rm Prym}(S_i,\Sigma)$ if and only if $\sigma^* L\cong L_i^*$,  the conditions that  $L_i\in {\rm Prym}(S_i,\Sigma)$ and $  L_i^2\cong \mathcal{O}$ are equivalent to the conditions that $\sigma^* L_i\cong L_i$ and $L_i^2\cong \mathcal{O}$.

 \begin{remark}  \label{remToledo1}

 The $SL(2,\R)$-Higgs bundles $(E_i,\Phi_i)$ have associated an integer invariant known as the Toledo invariant.
 This can be defined in several equivalent ways, including as the degree of the line bundle  {$N_i$}  or the Euler number of the $\SO(2)$-principal bundle associated to  {$N_i\oplus N_i^{*}$}. It can also be seen in the spectral data where it is detected by the action of the involution on the fibers of the line bundle at  fixed points (see  \cite{yoU}). However it is defined, the Toledo invariant is bound by a so-called Milnor-Wood inequality.
 \end{remark}


 \subsection{$\SO(4,\C)$-Higgs bundles and  $SO_0(2,2)$-Higgs bundles}\label{so022}   
  An $SO_0(2,2)$-Higgs bundle  can be described as  an $\SO(4,\C)$-Higgs bundle  {$(E,\Phi)$} where  {$E=W_1\oplus W_2$} decomposes as the sum of two rank 2 holomorphic oriented orthogonal bundles, and 
\begin{equation}\label{sonnphi}
\Phi=\left(\begin{array}{cc}
              0&\alpha\\
-\alpha^T&0
             \end{array}
\right),\end{equation}
for  $\alpha^T=q_2^{-1}\circ \alpha^{*}\circ q_1$ where $\alpha^{\rm T}$ is the dual map and $q_i$ is the orthogonal structure of $W_i$.  Furthermore,  Higgs bundles with structure group $\SO_0(2,2)$ have isomorphisms $\delta_i:\Lambda^2 W_i\cong \mathcal{O}$.   As in the previous section, we may take the rank two orthogonal bundles to be of the form
$  W_1=M_1\oplus M_1^{*}$
and $
W_2=M_2\oplus M_2^{*}$, 
for $M_i$ line bundles on $\Sigma$ with $\deg (M_i)\ge 0$, and with orthogonal structure as in \eqref{qstd}.
 

As in Eq.~\eqref{curvesonn}, the Higgs field $\Phi$ in an $\SO(4,\C)$-Higgs bundle defines a 4-fold cover $ \pi_{4}: S_4\rightarrow \Sigma$  with equation 
\begin{eqnarray}\label{P4}
P_4(\eta):=\det(\eta I-\Phi)=\eta^4+b_1\eta^2+b_2\eta+b_2^2=0,\label{PolQ4}
\end{eqnarray}
where  $b_2$ is the Pfaffian {,}  $b_i\in H^0(\Sigma, K^i)$, and $\eta$ the tautological section of $\pi^*K$.
The spectral data for an $SO_0(2,2)$-Higgs bundle is  {given by the spectral data $(\hat S_4,  L)$  of the corresponding complex $SO(4,\C)$-Higgs bundle}   for which $L^2\cong \CO$ \cite[Theorem 4.12]{thesis}.  Moreover, as in Section \ref{real so}, the direct image of $L$ in $\hat S_4/\hat \sigma$ defines two line bundles which induce  $W_i$ on $\Sigma$.


The group $\SO_0(2,2)$ is both a split real form of $\SO(4,\C)$ and a group of Hermitian type.  As a consequence of being a split real form, the Hitchin fibration admits a section which defines the Hitchin component in $\mathcal{M}_{\SO_0(2,2)}$. By virtue of the properties of groups of Hermitian type the Higgs bundles carry Toledo invariants, i.e. discrete invariants which in the case of $\SO_0(2,2)$-Higgs bundles may be taken to be the degrees of the line bundles $M_1$ and $M_2$.  The invariants are bounded by a Milnor-Wood  type inequality which in this case is (e.g. see \cite{bradlow-garcia-prada-gothen:2005} or \cite[Table C.2]{roberto})
\begin{equation}\label{MWSO_0(2,2)}
|\deg(M_i)|\le 2g-2.
\end{equation}

\subsection{The induced map on the Higgs bundles}\label{induced-Higgs-2}

\label{iso22higgs}
 From  Section \ref{Lie2} (see \cite{aparicio} for a detailed study of this), the map on $\SL(2,\C)\times\SL(2,\C)$- Higgs bundles induced by $\CI_2$ is 
 \begin{equation}\label{I*bundles}
 \CI_2((E_1,\Phi_1), (E_2,\Phi_2))=(E_1\otimes E_2,\Phi_1\otimes I+I\otimes\Phi_2),\end{equation}
 \noi where the orthogonal structure on $E_1\otimes E_2$ is the tensor product of the symplectic structures $\omega_i$ on $E_1$ and $E_2$ (recall that $\SL(2,\C)\simeq\mathrm{Sp}(2,\C)$).  
 
 \begin{remark}\label{detiso}Notice that the the isomorphisms $\det(E_i)\simeq\mathcal{O}_{\Sigma}$  do not uniquely determine a trivialization $\delta:\det(E_1\otimes E_2)\simeq \mathcal{O}_{\Sigma}$ compatible with the orthogonal structure on $E_1\otimes E_2$. Indeed, if $\{e_i^1,e_i^2\}$ are local oriented frames for $E_i$ satisfying $\omega_i(e_i^1,e_i^2)=1$, then both
 $\{e_1^1\otimes e^1_2,e_1^1\otimes e^2_2,e_1^2\otimes e^1_2,e_1^2\otimes e^2_2\}$ and $\{e_1^1\otimes e^1_2,e_1^2\otimes e^1_2,e_1^1\otimes e^2_2, e_1^2\otimes e^2_2\}$ are orthonormal local frames for $E_1\otimes E_2$ but they have opposite orientations.  They determine the two inequivalent choices for $\delta:\det(E_1\otimes E_2)\simeq \mathcal{O}_{\Sigma}$. 
\end{remark}

 {Following the notation of Section \ref{specsl2}, t}he map $\CI_2$ can then be put in the form
 \begin{align}\label{I*bundlesR}
 \CI_2((N_1\oplus N_1^{-1},\Phi_1), (N_2\oplus N_2^{-1},\Phi_2))&= \left((M_1\oplus M_1^{-1})\oplus (M_2\oplus M_2^{-1}) ,\left(\begin{array}{cc}0&\alpha\\ -\alpha^{\rm{T}}&0\end{array}\right)\right)  \\
\mathrm{where}\quad\quad  M_1=N_1\otimes N_2\ ,\  M_2&=N_1\otimes N^{-1}_2\ ,
\mathrm{and}\  \alpha=\begin{pmatrix}\beta_2&\beta_1\\\gamma_1&\gamma_2 \label{I*phi}
\end{pmatrix}.
 \end{align}

 \subsection{The induced map on spectral data} \label{iso22curve}
  {W}e  {shall} first construct the spectral data $(\hat S_4, \mathcal{L})$ associated to the  $SO(4,\mathcal{C})$-Higgs bundle obtained via $\CI_2$  {,and}  then specialize to the spectral data associated to the split real forms $\SL(2,\R)\times\SL(2,\R)$ and $\SO_0(2,2)$.

\begin{proposition} \label{propso21}
Let  $(S_i,L_i)$ be the spectral data corresponding to an $SL(2,\C)\times SL(2,\C)$-Higgs bundle, where each $S_i$ is defined by $\eta^2+a_i=0$ for $a_i\in H^0(\Sigma, K^2)$, and $L_i\in \Prym(S_i,\Sigma)$. Then the pair  $(\hat S_4, \mathcal{L})$ given by
\begin{itemize}  
\item  $\hat S_4:=S_1\times_{\Sigma}S_2$  is the fiber product curve, and 
\item $\mathcal{L}= p_1^*(L_1)\otimes p_2^*(L_2)$ is the line bundle 
\end{itemize}
 as in diagram \eqref{diagram}, gives  the spectral data associated to an $SO(4,\C)$-Higgs bundle.

  \end{proposition}
 \begin{remark}
 The total space of $K\oplus K$ can be identified with the fiber product $K\times_KK\subset K\times K$.  The fiber product $S_1\times_{\Sigma}S_2$ may thus be regarded as a subvariety of either $K\times_KK$ or of the total space of $K\oplus K$.
 \end{remark}
 \begin{proof} 
 The fiber product $\hat S_4$, as a curve in the total space of $K\oplus K$, is defined by the conditions
 \begin{equation}
 \eta_1^2+a_1=\eta_2^2+a_2=0,
 \end{equation}
\noi where $(\eta_1,\eta_2)$ denotes the tautological section on $K\oplus K$.  It follows that for a generic choice of $(a_1,a_2)$ this curve is smooth. Moreover, if $S_4\subset K$ denotes the image of $\hat S_4$ under the map
\begin{equation}\label{+map}
+:K\oplus K\rightarrow K
\end{equation}
given by fiberwise addition, then it is defined by the conditions
  \begin{equation}\label{S4eqtns}
 \eta=\eta_1+\eta_2\ ;\  \eta_1^2+a_1=\eta_2^2+a_2=0.
 \end{equation}
 {Hence,} the four-fold cover  $S_4$ is defined by the equation
\begin{equation}\label{S4curve} 
\eta^4+2(a_1+a_2)\eta^2+(a_1-a_2)^2=0,
\end{equation}
which from \S \ref{real so} and Eq.~\eqref{curvesonn}, is the spectral curve of an $SO(4,\C)$-Higgs bundle.
The curve $S_4$ is generically singular, with singularities over the zeros of $a_1-a_2$, and by construction the 
the map $+:\hat{S}_4\rightarrow S_4$ is  an isomorphism on the smooth locus of $S_4$.

   The involution $\sigma:\eta\mapsto -\eta$ which preserves $S_1$ and $S_2$  induces an involution $(\sigma ,\sigma )$ on $\hat{S}_4=S_1\times_\Sigma S_2$.  {When needed} we shall denote the involution on $S_1, S_2$ by $\sigma_i$, for $i=1,2$ and on $\hat{S}_4$ by $\hat{\sigma}_4$.  The fixed points of $\sigma_i$ are the zeros of $a_i$, and thus, since the zeros of $a_1$ and $a_2$ are generically different,  generically $\hat{\sigma}_4=(\sigma_1,\sigma_2)$ does not have any fixed points. It is clear from \eqref{S4eqtns} that $\hat{\sigma}_4$ descends to the involution $\eta\mapsto -\eta$ on the singular curve $S_4$, where it has fixed points at the branch locus.   
   
   In order to see that $\CL$ is the spectral line bundle associated to  {an} $SO(4,\C)$-Higgs bundle one has to show that  $\mathcal{L}\in\mathrm{Prym}(\hat{S}_4,\hat{S}_4/\hat\sigma_4)$.
 Since  $L_i\in{\rm Prym}(S_i,\Sigma)$ one has that $\sigma_i^*L_i\cong L_i^*$ and so the line bundle $\CL:=p_1^*(L_1)\otimes p_2^*(L_2)$  is sent to its dual by the involution $\hat{\sigma}_4$.  Hence, the line bundle $\CL$
on $\hat S_4$ is in ${\rm Prym}(\hat S_4,\hat S_4/\hat{\sigma}_4)$ as required. 
 \end{proof}

\begin{proposition}\label{I2spec} The spectral data $(\hat S_4,\CL)$ induced by an $SL(2,\C)\times SL(2,\C)$-Higgs bundle $(E_1,\Phi_1),(E_2,\Phi_2)$, as in Proposition \ref{propso21} corresponds to the spectral data of the $SO(4,\C)$-Higgs bundle $\CI_2[(E_1,\Phi_1),(E_2,\Phi_2)]$.
\end{proposition}

\begin{proof} As seen in Proposition \ref{propso21} the spectral curve $\hat S_4$ is the curve associated to an $SO(4,\C)$-Higgs bundle. Furthermore, from Eq.\eqref{S4curve} the curve is indeed the one associated to the Higgs bundle in the image of $(E_1,\Phi_1),(E_2,\Phi_2)$ through $\CI_2$. With the notation of \eqref{diagramR} below, in order to corroborate that the line bundle $\CL$ is indued the spectral line bundle of the image $SO(4,\C)$-Higgs bundle, note that for any line bundles  {$F_i$}
on $S_i$  {one has} 
 \begin{equation*}
\pi_*(p_1^*(F_1)\otimes p_2^*(F_2))=(\pi_1)_*(F_1)\otimes (\pi_2)_*(F_2).
\end{equation*}
\noi Applying this to the line bundles $L_i$ and using the relation \eqref{EfromL} with $n=2$,  we get
\begin{equation*}
\pi_*(\mathcal{L})=(\pi_1)_*(L_1)\otimes (\pi_2)_*(L_2)=E_1\otimes E_2\otimes K^{-1}\ ,
\end{equation*}
\noi whereas by \eqref{EfromL-SO} the vector bundle on $\Sigma$ defined by $\mathcal{L}$ is 
\begin{equation*}
E= \pi_*(\mathcal{L}\otimes (K_{\hat{S}_4}\otimes\pi^*K^*)^{1/2})\ .
\end{equation*}
\noi Recall that $K_{\hat{S}_4}\otimes\pi^*K^*$ corresponds to the ramification divisor $R\subset \hat{S}_4$, while the ramification divisors $R_i\subset S_i$ satisfy $[R_i]=\pi_i^*K$. It follows that
\begin{equation}
K_{\hat{S}_4}\otimes\pi^*K^*=[R]=p_1^*[R_1]\otimes p_2^*[R_2] = (\pi^*K)^2,
\end{equation}
 and hence that $E=\pi_*(\mathcal{L})\otimes K=E_1\otimes E_2$ as required.
\end{proof}

\subsection{The restriction to  $\SL(2,\R)\times\SL(2,\R)$}\label{realform22}
  {Let $(S_i,L_i)$ be the spectral data of an $\SL(2,\R)$-Higgs bundles, for $i=1,2$.}
  {T}hen one has that  $\CL^2\cong p_1^2(L_1)\otimes p_2^2(L_2)\cong \mathcal{O}$, i.e.
 \begin{eqnarray}\label{22prym}
 \CL\in P_{\hat{\sigma}_4}[2]:= \{M\in {\rm Prym}(\hat S_4,\hat S_4/\hat{\sigma}_4)~|~M^2\cong \mathcal{O}\}.\end{eqnarray}
\noi Thus as seen in Section \ref{real so} the line bundle $\mathcal{L}$ defines an $SO_0(2,2)$-Higgs bundle.

Since $\sigma_i^*{L_i}\simeq L_i^{-1}\simeq L_i$, it follows that $\hat{\sigma}_4^*{\mathcal{L}}\simeq\mathcal{L}^{-1}\simeq\mathcal{L}$. This means that under the projection $p:\hat{S}_4\rightarrow\hat{S}_4/{\hat{\sigma}_4}$ the direct image sheaf $p_*\mathcal{L}$ splits as the sum of two line bundles $\mathcal{L}_{\pm}$,   generated by $\hat{\sigma}_4$-invariant and anti-invariant local sections. The relation between the different covers of the Riemann surface and the line bundles on them is depicted in the following diagram:
\begin{eqnarray}\label{diagramR}
\xymatrix{
\CL=p_1^*L_1\otimes p_2^*L_2\ar[d]&p_*{\mathcal{L}}=\CL_+\oplus\CL_-\ar[d]\\
\hat{S}_4=S_1\times_{\Sigma}S_2\ar[d]^{\pi}\ar[r]^{p}&\hat{S}_4/{\hat{\sigma}_4}\ar[dl]^{\hat \pi}\\
\Sigma&
}
\end{eqnarray}
  In particular, this implies that
\begin{equation}
\pi_*{\mathcal{L}}=(\hat \pi)_*p_*\mathcal{L}=(\hat\pi )_*(\mathcal{L}_+\oplus\mathcal{L}_-).
\end{equation}
  Moreover,  $\hat{S}_4$ is the normalization of the spectral curve $S_4\subset \pi^*K$ and the involution $\hat{\sigma}_4$ on $\hat{S}_4$ corresponds to the involution $\eta\rightarrow-\eta$  on $S_4$. It follows that
multiplication by the tautological section $\eta$ interchanges $\mathcal{L}_+$ and $\mathcal{L}_-$, and thus that the Higgs field on $\pi_*(\mathcal{L})$ has the form as in \eqref{I*bundlesR}.  We thus get:

\begin{proposition}\label{I2realhiggs}Consider $((S_1,L_1),(S_2,L_2))$ the spectral data for a point in $\mathcal{M}_{\SL(2,\R)\times\SL(2,\R)}$ represented by 
$$\Bigg(\left(N_1\oplus N_1^{-1},\begin{bmatrix}0&\beta_1\\\gamma_1&0\end{bmatrix}\right)\ ,\ \left(N_2\oplus N_2^{-1},\begin{bmatrix}0&\beta_2\\\gamma_2&0\end{bmatrix}\right)\Bigg),$$

\noi  and let $(\hat S_4,\mathcal{L})$ be defined as in Proposition \ref{propso21}. Then $(\hat S_4,\mathcal{L})$  is the spectral data for the point in $\mathcal{M}_{\SO_0(2,2)}$ represented by 

$$\Bigg((N_1N_2\oplus (N_1N_2)^{-1})\oplus N_1N^{-1}_2\oplus (N^{-1}_1N_2)), \begin{bmatrix}0&\alpha\\-\alpha^{\rm{T}}&0\end{bmatrix}\Bigg)$$
\noi where $\alpha$ is as in Eq. \eqref{I*phi}.
\end{proposition}


%


\section{The rank 3 isogeny and  the Hitchin fibration}\label{isosmooth}

In this section we investigate the   induced map  {
$\label{i3}\CI_3:\CM_{SL(4,C)}\rightarrow \CM_{SO(6,\C)},$} 
and its restriction to the split real forms $\SL(4,\R)$ and $\SO_0(3,3)$.

 \subsection{$SL(4,\C)$-Higgs bundles and $SL(4,\R)$-Higgs bundles}\label{rec1} 
  {From} Definition \ref{real form}, an $\SL(4,\R)$-Higgs bundle on $\Sigma$ is holomorphic $\SO(4,\C)$-principal bundle together with a symmetric Higgs field. Equivalently it can be viewed as a pair $(E, \Phi)$ where $E$ is an oriented holomorphic rank 4 orthogonal vector bundle, i.e. a vector bundle with a holomorphic symmetric non-degenerate bilinear paring $Q$, and a compatible isomorphism $\delta:\det(E)\simeq\mathcal{O}$, and the Higgs field $\Phi:E\rightarrow E\otimes K$ is traceless and symmetric with respect to $Q$.
  
  Recall from Section \ref{rec1}  that the spectral curve for an $\SL(4,\C)$-Higgs bundle $(E,\Phi)$ is a ramified 4-fold cover $\pi:S\rightarrow \Sigma$  in the total space of $K$  with equation  
\begin{eqnarray}\label{P4}
P_4(\eta):=\det(\eta I-\Phi)=\eta^4+a_2\eta^2+a_3\eta^3+a_4=0,\label{PolP}
\end{eqnarray}
for $a_i\in H^0(\Sigma, K^i)$ and $\eta$ the tautological section of $\pi^*K$.  For generic choices of $\{a_2,a_3,a_4\}$ the curve is smooth and has only the most generic ramification, i.e. in fibers over the branch locus there are two unramified points and one order two ramification point.  The spectral data is completed by a 
line bundle  {$L\in {\rm Prym}(S,\Sigma)$.}  
For an $\SL(4,\R)$-Higgs bundle, from \cite[Theorem 4.12]{thesis} the spectral  line bundle $L\in {\rm Prym}(S,\Sigma)$ satisfies the extra condition $L^2\cong \CO$.

 
\subsection{$SO(6,\C)$-Higgs bundles and $SO_0(3,3)$-Higgs bundles}
An $SO_0(3,3)$-Higgs bundle  can be described as  $\SO(6,\C)$-Higgs bundle   $(E,\Phi)$where $E$ decomposes as the sum of two rank 3 holomorphic oriented orthogonal bundles, say $E=W_1\oplus W_2$ with orthogonal structures $q_i$ and compatible isomorphisms $\delta_i:\Lambda^3 W_i\cong \mathcal{O}$.  As seen in Sections \ref{season}-\ref{real so}, for any $\SO(6,\C)$-Higgs bundle the Higgs field $\Phi$  defines a 6-fold cover $ \pi_{6}: S_6\rightarrow \Sigma$  with equation 
\begin{eqnarray}\label{P6}
P_6(\eta):=\det(\eta I-\Phi)=\eta^6+b_1\eta^4+b_2\eta^2+b_3^2=0,\label{PolQ}
\end{eqnarray}
where  $b_3$ is the Pfaffian and $b_i\in H^0(\Sigma, K^i)$. The spectral data of an $SO(6,\C)$-Higgs bundle is then a pair $(\hat{S}_6, L_6)$ where $\hat{S}_6$ is the desingularization of $S_6$ and $L_6\in {\rm Prym}(\hat S_6,\hat S_6/ \hat\sigma)$ where $\hat{\sigma}$ is the (fixed-point-free) involution inherited from $S_6$ \cite{N2}.
When $L_6^2\cong \CO$ the spectral data corresponds to an $SO(3,3)$-Higgs bundle.
   
\subsection{The induced map on Higgs bundles}\label{Higgs33}

%
%
%
Using the results in Section  \ref{Lie1} for the induced action of $\CI_3$ on vector bundles and Lie algebras, we get the map between complex Higgs bundles
\begin{equation}\label{I3onHiggs}
\CI_3(E,\Phi)= (\Lambda^2E, \Phi\otimes I+I\otimes\Phi)\ .
\end{equation}
The orthogonal structure $Q:\Lambda^2E\rightarrow \Lambda^2 E^*$ is induced by the combination of \eqref{QforLambda2} and the trivialization of $\det(E)$.   As in the case of the map defined by $\CI_2$ (see Remark \ref{detiso})  an isomorphism $\delta: \det(\Lambda^2E)\rightarrow \mathcal{O}_{\Sigma}$ such that $\delta^2$ agrees with the trivialization of $(\det(\Lambda^2E))^2$ determined by $Q$, is determined only up to a choice of sign. There are thus two (oppositely oriented) possible conventions for determining the $\SO(6,\C)$ structure on $\Lambda^2E$.  This choice plays a role in the map induced by $\CI_3$ on the base of the Hitchin fibrations of the moduli spaces (see Section \ref{basemap3}).

 The vector bundle $E$  {of an $\SL(4,\R)$-Higgs bundle $(E,\Phi)$} has an oriented orthogonal structure, i.e. an associated pair $(q,\epsilon)$  where $q$ is  a holomorphic orthogonal structure on $E$, and $\epsilon$ is a compatible isomorphism $\epsilon:\det(E)\simeq \mathcal{O}$ trivializing its determinant.  The orthogonal structure induces  an isomorphism (by abuse of notation, also denoted by $q$)
\begin{equation}\label{Lambda2Q}
q:\Lambda^2E\rightarrow\Lambda^2E^*.
\end{equation}
  Using $\epsilon$ as the trivialization of $\det(E)$ required in the construction of $Q$,  yields an isomorphism
\begin{equation} *=q^{-1}\cdot Q:\Lambda^2E\rightarrow\Lambda^2E,
\end{equation}
\noi  which satisfies\footnotemark\footnotetext{We have denoted the involution by $*$ since when $E$ is the cotangent bundle to a 4-manifold, the involution is precisely the Hodge star.}
\begin{equation}
q(\alpha ,\beta)=Q(\alpha,*\beta),
\end{equation}
\noi where   $q,Q$ are regarded as bilinear forms on $\Lambda^2E$, and  $\alpha,\beta\in \Lambda^2E$.    By using a local oriented orthonormal frame to compute $*$, it can be seen  that $*$ satisfies $*^2=I$.  Taking the $\pm 1$ eigenspaces of $*$ thus gives a decomposition
\begin{equation}\label{decomp}
\Lambda^2E=\Lambda^2_+ E\oplus\Lambda^2_-E.
\end{equation}

\noi The orthogonal structure on $\Lambda^2E$ restricts to orthogonal structures on each summand, so that the structure group reduces to $\SO(3,\C)\times\SO(3,\C)$. 
  
\begin{remark}
The structure groups of the bundles $\Lambda^2E$ and $\Lambda^2_+ E\oplus\Lambda^2_-E$ can be reduced to $\SO(4)$ and $\SO(3)\times\SO(3)$ respectively. The two copies of $\SO(3)$ are precisely the normal subgroups mentioned in Remark \ref{notsimple} by virtue of which $\SO(4)$ fails to be simple (see \cite{Go2}) .
\end{remark}

With respect to the reduction in \eqref{decomp}, the Higgs field $\Phi\otimes I+I\otimes\Phi$ has the form in \eqref{sonnphi}, where $\Phi$  and  $\alpha$ are related as in Eq.~\eqref{alphamatrix}  and $-\alpha^t$ is the orthogonal transpose.  Denoting orthogonal structures on $\Lambda^2_{\pm}E$ by $q_{\pm}$, we thus get:

\begin{proposition}\label{realI3} The isogeny $\mathcal{I}_3$ induces the following map between $\SL(4,\R)$- {Higgs bundles} and $\SO_0(3,3)$-Higgs bundles: 
\begin{equation}\label{SL4onHiggs}
(E,\Phi)\mapsto \left(\Lambda^2_+E\oplus\Lambda^2_-E, \begin{bmatrix}0&\alpha\\ -\alpha^{\rm{T}}&0\end{bmatrix}\right),
\end{equation} 
\noi where if $E$ has orthogonal structure $q$ then the bundles $\Lambda^2_{\pm}E$ have orthogonal structures  $q_{\pm}$, and $\Phi$  and  $\alpha$ are related as in Eq.~\eqref{I3algebrasym}.
\end{proposition}

\begin{remark}The precise form of $\alpha$  {depends}  on the choices  {of} the orientations of $\Lambda^2E$ and $\Lambda^2_{\pm}E$. Different choices will change the sign of $\det{\alpha}$, i.e. of the Pfaffian of ${\tiny \begin{bmatrix}0&\alpha\\ -\alpha^{\rm{T}}&0\end{bmatrix}}${.}
\end{remark}

\subsection{The induced map on spectral  data} \label{specral33} 

 Given $\SL(4,\C)$-spectral data $(S,L)$,  with $S$ defined by \eqref{P4} and $L\in\Prym(S,\Sigma)$, we build $\SO(6,\C)$-spectral data using a construction similar to the fiber product construction in Section \ref{iso22curve}, except in this case we take the product of $(S,L)$ with itself, i.e. in diagram \eqref{diagram} we have $S_1=S_2$ and $L_1=L_2$.  The resulting curve has both singularities and additional symmetries that are absent when $S_1$ and $S_2$ are different. Our construction takes both of these features into account in an essential way.

 The curve $S\times_{\Sigma}S$ is a 16-fold cover of the Riemann surface $\Sigma$.  Over a generic point in the Hitchin base, $S$ is smooth and $S\times_{\Sigma}S$ has two smooth components, namely the diagonal
$S_{\Delta}:=\{(s,s)\in S\times_{\Sigma}S\}$ and another one  which we denote by $(S\times_{\Sigma}S)_0$.  The intersection of these components lies in fibers over the branch locus of the covering $\pi:S\rightarrow\Sigma$.

Viewing the curve $S\times_\Sigma S$ in the total space of $K\oplus K$,  the involution $\tau:(x,y)\mapsto (y,x)$ interchanges the copies of $S${,} and thus the fixed point set of $\tau$ is $S_{\Delta}$.  
The quotient map \linebreak
$\pi_\tau:(S\times_{\Sigma}S)_0\rightarrow (S\times_{\Sigma}S)_0/{\tau}$
commutes with the projection onto $\Sigma$. It is an unramified double cover on $S\times_{\Sigma}S-S_{\Delta}$ but has ramification points in the fibers over the base locus of $\pi:S\rightarrow\Sigma$. Using the biholomorphism $Sym:(K\oplus K)/{\tau}\rightarrow K\oplus K^2$  given by
 \begin{eqnarray}
 Sym:(x,y)\mapsto \left(\frac{x+y}{2}, xy\right), \label{symtrans}
 \end{eqnarray}
we can view the quotient $(S\times_{\Sigma}S)_0/{\tau}$ as a curve in the total space of $K\oplus K^2$.  We define 
\begin{equation}\label{hatS6}
\hat{S}_6:=Sym((S\times_{\Sigma}S)_0/{\tau}),
\end{equation}
\noi  and denote by $\hat\pi_\tau$ the composition of $\pi_\tau$ and $Sym$. A depiction of the relation between the above curves and projections is given  in diagram  \eqref{diagram46} below.  By abuse of notation, we denote by $p_i:(S\times_{\Sigma}S)_0\rightarrow\Sigma $  the restrictions to $(S\times_{\Sigma}S)_0$ of the projection maps to the two factors of the full fiber product:

\begin{eqnarray}\label{diagram46}
\xymatrix{
S\ar@<-1ex>[dr]_{\pi}&\ar@<-.45ex>[l]_{p_1} \ar@<.45ex>[l]^{p_2}(S\times_{\Sigma}S)_0\ar[d]_{\pi_{0}}  \ar[r]^{\hat{\pi}_{\tau}}&\hat{S}_6\ar@<1ex>[ld]^{\hat{\pi}_6}\\
 &\Sigma&&&&
}
\end{eqnarray}

 \begin{lemma}  \label{uno2}For generic points in the Hitchin base, the ramification divisors $R$ on $S$, $R_{0}$ on $(S\times_{\Sigma}S)_0$  and $\hat{R}_6$ on $\hat{S}_6$ for the projections $\pi, \hat{\pi}_{\tau}$ and $\hat{\pi}_6$ respectively,  are related as follows:
\begin{equation}\label{pull}
p_1^{-1}(R)+p_2^{-1}(R)=(\hat{\pi}_\tau)^{-1}(\hat{R}_6)+2R_{0}.
\end{equation}
\end{lemma}

\begin{proof} The ramification divisor for a covering $\pi:X\rightarrow Y$ is defined by $R=\Sigma_{y\in Y} R(y)\cdot y$ where the weights $R(y)$ are such that $R(y)+1$ is the multiplicity of $\pi$ at $y$.  If $S$ is a generic spectral curve and $x$ is a point in the branch locus of $\pi:S\rightarrow\Sigma$ then we can write $\pi^{-1}(x)=\{y_1,y_2,y_3\}$, with $y_1\in R$ (with weight 1) but $\pi$ unramified at $y_2,y_3$.  Then the fiber of $(S\times_{\Sigma}S)_0/{\tau}$ over $x$ consists of points $ \{[y_1,y_1], [y_1,y_2],[y_1,y_3],[y_2,y_3]\}$.  Of these, $Sym[y_1,y_2]$ and $Sym[y_1,y_3]$ land in $\hat{R}_6$, each with weight 1, while $\hat{\pi}_{\tau}^{-1}Sym([y_1,y_1])=(y_1,y_1)$ is in $R_0$ (with weight 1).  Notice now that in the fiber over $x$ one has
\begin{align*}
\pi_0^{-1}(x)\cap( p_1^{-1}(R)+p_2^{-1}(R))&=\ 2(y_1,y_1)+(y_1,y_2)+(y_2,y_1)+(y_1,y_3)+(y_3,y_1)\\
&=\ 2(y_1,y_1)+\hat{\pi}_{\tau}^{-1}(Sym([y_1,y_2])+Sym([y_1,y_3]))
\end{align*} as required.
\end{proof}

\begin{proposition}\label{uno1}
The curve $\hat S_6$ is generically smooth, and the canonical desingularization of its projection to $K$ through the $Sym$ map, which is the spectral curve of an $SO(6,\C)$-Higgs bundle. 
 \end{proposition}

\begin{proof} In order to prove the proposition, one needs to show the following hold: {
\begin{enumerate}\item the curve $\hat{S}_6$ is generically smooth;
\item under the projection $q_1:K\oplus K^2\rightarrow K$ which on each fiber is given by $(u,v)\mapsto 2u$, the image $S_6 := q_1(\hat{S}_6)$ is a spectral curve defined by an equation of the form in \eqref{P6} with 
\begin{equation}\label{bvsa}
b_1=2a_2,   b_2= a_2^2-4a_4, b^2_3=a^2_3;
\end{equation}
\item the projection $q_1:\hat{S}_6\rightarrow S_6$ is an isomorphism away from the singularities of $S_6$ at its intersection  with the zero section of $K$.
\end{enumerate}}
 {To prove the above items,} let  $S\subset K$ be defined by  {the zero locus of \eqref{P4}}.  As a curve in the total space of $K\oplus K$, the fiber product $S\times_{\Sigma}S$ is defined by the conditions   {$P_4(\eta_1)=P_4(\eta_2)=0$}, where $\eta_1,\eta_2$ denote  the tautological sections of the two summands in $K\oplus K$. 
 After using the transformation $Sym$ defined in  \eqref{symtrans} to realize the curve as a subvariety of $K\oplus K^2$, the component  $\hat{S}_6$ is described locally (i.e. with respect to a trivialization of the bundles) as the zero locus of the map
 $\mathcal{F}:\C^3 \rightarrow \C^2$ given by
\[(z,u,v)\mapsto (8u^3 -4uv + 2a_2(z)u + a_3(z), 
 8u^4  + 2a_2(z)u^2 - 8u^2v - a_2(z)v  + a_3(z)u + v^2 + a_4(z)).\]
 The projection $q_1$ onto the total space of $K$ is given locally by $(z,u,v)\mapsto (z,2u)$. Direct computation shows that away from $u=0$, the map is a biholomorphism onto the curve define by the equation
\begin{equation}
P_6(z,\eta)= \eta^6  + 2a_2(z)\eta^4 +(a_2(z)^2- a_4(z))\eta^2+a_3^2=0,
\end{equation}
where $\eta$ is the local fiber coordinate on $K$.
This curve, which we denote by $S_6$, necessarily has singularities  {at the zeros of $a_3$,}  but is otherwise smooth for generic choice of $a_2,a_3,a_4$.  For such choices,  singularities in $\hat{S}_6$ can occur only at points where $u=0$. Direct computation shows that the derivative of $\mathcal{F}$ has full rank at such points provided $a_3$ and $a_2^2- a_4$ have no common zeros,  {proving the proposition.}
\end{proof}

 {
\begin{lemma}
The involution $\eta\rightarrow-\eta$ on $S_6$ lifts to the involution $\sigma:\hat{S}_6\rightarrow\hat{S}_6$ given on fibers away from the branch locus by the map $Sym[y_i,y_j]\mapsto Sym[y_k,y_l]$, where $\{i,j,k,l\}=\{1,2,3,4\}$. On a fiber with a ramification point, say $y_1$, the involution maps $Sym[y_1,y_1]$ to $Sym[y_2,y_3]$. 
\end{lemma}}
\begin{proof}
On a  {regular} fiber away from the base locus,  where we can write $\pi^{-1}=\{y_1,y_2,y_3,y_4\}$, the coordinate $\eta$ on $S_6$ has values $y_i+y_j$ for $i\ne j$. But $y_1+y_2+y_3+y_4=0$ and hence $-(y_i+y_j)=(y_k+y_l)$ where $\{i,j,k,l\}=\{1,2,3,4\}$, i.e. $\eta\mapsto -\eta$ corresponds to the action of $\sigma$. Assuming that only the most generic type of ramification occurs, the computation is similar on fibers over the branch locus.
\end{proof}

To obtain the  {correspondence induced by $\CI_3$ between} spectral line bundles we begin as in Section \ref{iso22curve}, but with $L_1=L_2$, and consider 
\begin{equation}\label{mathcalL}
\mathcal{L}=p_1^*(L)\otimes p_2^*(L)\ .
\end{equation}
\noi  Since $\mathcal{L}$  is invariant under $\tau${$:(x,y)\mapsto (y,x)$},  its direct image on $\hat S_6$
decomposes as a sum of rank one locally free sheaves generated by the invariant and anti-invariant local sections, i.e.
  \begin{equation} 
 (\hat{\pi}_{\tau})_*\mathcal{L}= \CL_+\oplus \CL_- .
  \label{lpm}
  \end{equation}
 \noi Similarly
 \begin{eqnarray}(\hat{\pi}_{\tau})_*\mathcal{O}_{(S\times_{\Sigma}S)_0}=\mathcal{O}_{\hat{S}_6}\oplus T.\end{eqnarray}
\noi  Moreover, if $[R_0]$ is  the line bundle on $(S\times_{\Sigma}S)_0$ defined by the divisor $R_0$, then  $(\hat{\pi}_{\tau})^*T=[R_0]^{-1}$ (see for example \cite[Lemma 3.1]{manetti} or \cite[p. 49]{catanese}). We define
 \begin{equation}\label{I3L}
 \mathcal{I}_3(L):=\mathcal{L}_-\otimes T^{-1}.
 \end{equation}
 
 \noi  \begin{remark}  While we have defined $\mathcal{L}$ only on $(S\times_{\Sigma}S)_0\subset S\times_{\Sigma}S$, this distinction disappears in $\CL_-$. This is a consequence of the fact that anti-invariant local sections must vanish on the fixed points of $\tau$, so that the sheaf on $(S\times_{\Sigma}S)/{\tau}$ generated by the anti-invariant sections has support only on $(S\times_{\Sigma}S)_0/{\tau}$.
\end{remark}

 
  \begin{remark}\label{LandM}  Since $\mathcal{L}$ is clearly invariant under pullback by the involution $\tau$ and admits a lift of $\tau$ which acts as identity on fibers over all fixed points, it follows (see for example \cite{nevins}) that $\mathcal{L}=\hat{\pi}_{\tau}^{-1}\mathcal{M}$ for some line bundle $\mathcal{M}$ on $\hat{S}_6$. We thus get 
  \begin{equation*}
  (\hat{\pi}_{\tau})_*{\mathcal{L}}=(\hat{\pi}_{\tau})_*((\hat{\pi}_{\tau})^*\mathcal{M})=\mathcal{M}\oplus(\mathcal{M}\otimes T),
  \end{equation*}
  
  \noi with the summands $\mathcal{M}$ and $\mathcal{M}\otimes T$ generated by the $\tau$-invariant  and anti-invariant local sections of $(\hat{\pi}_{\tau})_*(\hat{\pi}_{\tau})^*\mathcal{M}$ respectively.  With this choice of $\mathcal{M}$ we get
 \begin{equation}\label{I3LM}
 \mathcal{I}_3(L)=\mathcal{M}\ .
 \end{equation}
  \end{remark}

\begin{proposition}\label{PrymMap} If  {$L \in \Prym(S,\Sigma)$} then  {$\mathcal{I}_3(L)\in \Prym(\hat{S}_6,\hat{S}_6/{\sigma})$.} \end{proposition}

\begin{proof}  {Let} $L$  {be}  defined by the divisor $D$ on $S$ given by
$$D=\sum_{x\in\Sigma}\sum_{y\in\pi_0^{-1}(x)}D(y)\cdot y:=\sum_{x\in\Sigma}D_x.$$
 {For}  $\mathcal{D}$ the divisor $p_1^{*}(D)+p_2^{*}(D)$ on $(S\times_{\Sigma}S)_0${, one has that}  $\mathcal{D}$ is in the linear system of $\mathcal{L}$.   {Moreover, when} $x$  {is} not in the branch locus of $\pi_0$, so that $\pi_0^{-1}(x)$ has four distinct points $\{y_1,y_2,y_3,y_4\}$,  {one has that} %
\begin{equation}\label{Cx}
\mathcal{D}_x :=\mathcal{D}\cap\pi_0^{-1}(x)=\hat{\pi}_{\tau}^{*}\Big(\sum_{i\ne j}(D(y_i)+D(y_j))\cdot Sym([y_i,y_j])\Big):=\hat{\pi}_{\tau}^{*}\mathcal{C}_x.
\end{equation}
   {On the other hand, i}f $x$ is a branch point,  then under our genericity assumptions on $S${,} we can assume that $\pi_0^{-1}(x)$ has one point (say $y_1$) where $\pi_0$ has multiplicity two, and two points (say $y_2,y_3$) where $\pi_0$ is unramified. Then $\mathcal{D}_x= \hat{\pi}_{\tau}^{*}\mathcal{C}_x$ with
\begin{equation}\label{CxR}
\mathcal{C}_x=D(y_1)\cdot Sym[y_1,y_1]+(D(y_2)+D(y_3))\cdot Sym[y_2,y_3]+[\sum_{i=2,3}(D(y_1)+2D(y_i))\cdot Sym[y_1,y_i]
\end{equation}
\noi where the factor 2 in the last term comes from the fact that the projection $p_1$ is ramified at the points $(y_1,y_2)$ and $(y_1,y_3)${,} while $p_2$ is ramified at the points $(y_2,y_1)$ and $(y_3,y_1)$.  {Moreover, from Remark \ref{LandM} one has that $\mathcal{L}=\hat{\pi}_{\tau}^{*}(\mathcal{M})$ and $\mathcal{I}_3(L)=\mathcal{M}$, and hence} $\mathcal{C}$  {is} defined by \eqref{Cx} {,} and \eqref{CxR} is a divisor in the linear system of $\mathcal{I}_3(L)$. Notice  that on fibers over branch points of $\pi_0$  {one has}  $\sigma[y_1,y_1]=[y_2,y_3]$.  {Then, d}enoting the branch locus of $\pi_0$ by $B$,  {one gets} %
\begin{equation}
Nm_{\sigma}(\mathcal{C}_x)=\begin{cases}(D(y_1)+D(y_2)+D(y_3)+D(y_4))\cdot\sum_{i=2}^4 [[y_1,y_i]]\ \mathrm{if}\ x\in \Sigma-B; \\
(D(y_1)+D(y_2)+D(y_3))\cdot ([[y_1,y_1]]+2[[y_1,y_2]])\ \mathrm{if}\ x\in B,
\end{cases}\label{labelB}
\end{equation}
\noi where $[[y_1,y_2]]$ denotes the point in $\hat{S}_6/{\sigma}$ and  $Nm_{\sigma}$ is the norm map for the covering $\hat{S}_6\rightarrow\hat{S}_6/{\sigma}$.  Finally, note that if $L\in Prym(S,\Sigma)$ then we can pick $D$ so that $Nm_{\pi_0}(D_x)=0$ for all $x\in \Sigma$ and hence $Nm_{\sigma}(\mathcal{I}_3(L))=Nm_{\sigma}(\mathcal{C})=0$, as required. 
 \end{proof}

We have shown that the map $(S,L)\mapsto(\hat{S}_6,\mathcal{I}_3(L))$ sends spectral data for an $\SL(4,\C)$-Higgs bundle to $\SO(6,\C)$-spectral data.  We now show that this map is compatible with 
the map given by \eqref{I3onHiggs}.

     
 \begin{proposition}\label{rank3main} The spectral data $(\hat S_6,\mathcal{I}_3(L))$ induced by an $SL(4,\C)$-Higgs bundle $(E,\Phi)$ via \eqref{hatS6} and \eqref{I3L} corresponds to the spectral data of the $SO(6,\C)$-Higgs bundle $\CI_3[(E,\Phi)]$.
 \end{proposition}   
 
 \begin{proof} 
By \eqref{I3onHiggs} the Higgs field in  $\CI_3[(E,\Phi)]$ is $\Phi_6:=\Phi\otimes I + I\otimes\Phi$, viewded as a map on $\Lambda^2E\subset E\otimes E$.   Let  %
$\eta^4+a_2\eta^2+a_3\eta+a_4$ and $\eta^6+b_2\eta^4+b_4\eta^2+b_3^2
$ be the characteristic polynomials for $\Phi$ and $\Phi_6$  respectively.   Then a calculation based on Remark \ref{eigensum} shows that the coefficients are related by \eqref{bvsa}.  It follows that the curve $\hat{S}_6$ is the spectral curve for the $SO(6,\C)$-Higgs bundle $\CI_3[(E,\Phi)]$.   Moreover, Proposition \ref{uno1} then shows that $\Phi_6$ can be recovered by pushing down the section $(\eta,\eta^2)${,} where $\eta$ is tautological section of $K$.

The direct image $\pi_*(\mathcal{I}_3(L)\otimes [\hat{R}_6]^{1/2})$ acquires a special orthogonal structure in the usual way, with the trivialization of its determinant bundle determined by the Prym condition $\sigma^*(\mathcal{I}_3(L))\simeq\mathcal{I}_3(L)^*$.  In order to see that $\pi_*(\mathcal{I}_3(L)\otimes [\hat{R}_6]^{1/2})=\Lambda^2E$, consider the (singular) full fiber product $S\times_{\Sigma}S$ and for any line bundle $N$ on $S$ let $\tilde{\mathcal{N}}=p_1^*(N)\otimes p_2^*(N)$, where now the projections are from the entire $S\times_{\Sigma}S$.  Arguing as before, since $\tilde{\mathcal{N}}$ is invariant under the involution $\tau$, it follows that its direct image under the projection $\pi_{\tau}:S\times_{\Sigma}S\rightarrow (S\times_{\Sigma}S)/{\tau}$ decomposes as the sum of rank one coherent sheaves  {$\tilde{\mathcal{N}}_-\oplus\tilde{\mathcal{N}}_+$}.  Though in principle these summands need not be locally free,  their restriction to $(S\times_{\Sigma}S)_0$ is well-behaved.  Indeed, given  the  projections $\pi:S\rightarrow \Sigma$ and $\hat\pi:(S\times_{\Sigma}S) /\tau\rightarrow \Sigma$ one gets    \begin{small}
\begin{eqnarray}\label{plusminus}
\hat\pi_*(\tilde{\mathcal{N}}_+)\oplus\hat\pi_*(\tilde{\mathcal{N}}_-)=
Sym^2(~\pi_*(p_1^*(N)\otimes p^*_2(N))~)\oplus\Lambda^2(~\pi_*(p_1^*(N)\otimes p^*_2(N))~).
\end{eqnarray}\end{small}
%
\noi But the left- and right-hand sides of this identification are both $\pm1$-eigenspace decompositions for $\Z_2$-actions that are compatible with the projection map.  We can thus identify $\hat\pi_*(\tilde{\mathcal{N}}_-)\simeq \Lambda^2(\pi_*(p_1^*(N)\otimes p^*_2(N))$.  Notice, furthermore, that the support of $\tilde{\mathcal{N}}_-$ is $(S\times_{\Sigma}S)_0$, so that $\tilde{\mathcal{N}}_-=\mathcal{N}_-$ where dropping the tilde on $\tilde{\mathcal{N}}$ denotes the restriction to $(S\times_{\Sigma}S)_0$. It follows that  
\begin{equation}\label{anyN}
\hat{\pi}_*\mathcal{N}_-=\Lambda^2({\pi}_*N )\ .
\end{equation}
\noi In particular, taking $N=L\otimes [R]^{1/2}$ where $R$ is the ramification divisor on $S$, so that $\pi_*N=E$, it follows that \eqref{anyN} gives $\hat{\pi}_*\mathcal{N}_-=\Lambda^2 E$.
It thus remains to show that $\mathcal{N}_-=\mathcal{I}_3(L)\otimes [\hat{R}_6]^{1/2}$. But by Lemma \ref{uno2} and \eqref{pull} 
\begin{align*}
\mathcal{N}&=\mathcal{L}\otimes [p_1^{-1}(R)+p_2^{-1}(R)]^{1/2}\\
&= \pi_{\tau}^{*}(\mathcal{M}\otimes[\hat{R}_6]^{1/2}\otimes T^{-1}),
\end{align*}
\noi where $\mathcal{M}$ is as in Remark \ref{LandM} and $T$ is as above. It follows that 
\begin{equation}
\mathcal{N}_-= \mathcal{M}\otimes[\hat{R}_6]^{1/2}\otimes T^{-1}\otimes T = \mathcal{I}_3(L)\otimes [\hat{R}_6]^{1/2},
\end{equation} as required.
\end{proof}

 %
 %
 


We  {shall} end this section with a discussion of the relation between our construction and the so-called trigonal construction of Recillas (see \cite{ron} or \cite{recillas}).  Given a smooth curve $\Sigma$, this construction relates a smooth four-fold cover of  $\Sigma$ to a six-fold cover with a fixed-point-free involution {,} whose quotient is thus a three-fold smooth cover.  Taking $S$ as the four-fold cover, the curve  {$\mathbb{S}:=$}$Sym^{-1}(\hat{S}_6)=(S\times_{\Sigma}S)_0/{\tau}$  is precisely the corresponding six-fold cover with involution.   This is most easily seen by considering the fibers of the covering maps onto $\Sigma$.  At the regular fibers where, if the fiber of $S$ over a point $x\in\Sigma$ consists of points $\{y_1(x),y_2(x),y_3(x),y_4(x)\}$ then the fiber of $(S\times_{\Sigma}S)_0/{\tau}$ over $x$ consists of the points corresponding to the unordered pairs $\{[y_1,y_2],[y_1,y_3],[y_1,y_4],[y_2,y_3],[y_2,y_4], [y_3,y_4]\}$ (where we have dropped the dependence on $x$ to simplify the notation).  

 {Considering} $\pi:S\rightarrow \Sigma$   a smooth cover with only the most generic ramification, if $x$ is in the branch locus and $\pi^{-1}(x)=\{y_1,y_2,y_3\}$ (as in the proof of Proposition \ref{PrymMap}) {,} then the fiber of $\mathbb{S}$  consists of the unordered pairs $\{[y_1,y_1],[y_1,y_2],[y_1,y_3],[y_2,y_3]\}$.  This relation between $S$ and $\mathbb{S}$ can be regarded as a map from $S$ to the  {third}  symmetric product of $\mathbb{S}$, mapping a point in $S$ to the unordered pairs containing that point.  This determines a correspondence defined by an effective divisor on  $S\times \mathbb{S}$, but for the sake of comparison with our construction, we use the map $Sym$ to replace $\mathbb{S}$ with $\hat{S}_6$.

\begin{definition} Define the effective divisor $\Delta\subset S\times \hat{S}_6$ by

\begin{equation}\label{Delta}
\Delta=\sum_{\tiny y\in S, ~[y,y']\in\mathbb{S}}\big(y,Sym[y,y']\big).
\end{equation}
\end{definition}

  Notice that $\Delta$ is contained in $S\times_{\Sigma}\hat{S}_6$ and is given explicitly by
\begin{align*}
\Delta=&\sum_{x\in \Sigma-B}\sum_{i=1}^4\sum_{j\ne i}\Big(y_i(x), Sym[y_i(x),y_j(x)]\Big)\\
&+\sum_{x\in B}\Bigg(\sum_{j=1}^3(y_1(x),Sym[y_1(x),y_j(x)]\Big)+\sum_{i=2}^3\sum_{j\ne i}\Big(y_i(x),Sym[y_i(x),y_j(x)]\Big)\Bigg).
\end{align*}
\noi Here  {the set} $B\subset\Sigma${, as in \eqref{labelB},} is the branch locus of $\pi${,} and at $x\in B$ we label the points in the fiber so that $\pi$ has degree 2 at $y_1(x)$.  

\begin{remark}  {The map} 
\begin{align*}
C:(S\times_{\Sigma}S)_0&\rightarrow S\times \hat{S}_6\\
(y_i(x),y_j(x))&\mapsto \Big(y_i(x),Sym[y_i(x),y_j(x)]\Big),
\end{align*}
 is injective and its image is precisely the divisor $\Delta$.  
\end{remark}

\noi Using the correspondence defined by $\Delta$ we  {may define} 
\begin{align*}
C_{\Delta}:Jac(S)&\rightarrow Jac(\hat{S}_6)\\
[D]&\mapsto [Nm_2(\pi_1^{*}(D)\cap\Delta)]
\end{align*}
 {for $\pi_1$   the projection   $S\times\hat{S}_6\rightarrow S$, and the norm map $Nm_2: Jac(S\times\hat{S}_6)\rightarrow Jac(\hat S_6)$.}
    
\begin{proposition} The map $C_{\Delta}$  restricts to a map between Prym varieties, i.e.
\[C_{\Delta}:Prym(S,\Sigma)\rightarrow Prym(\hat{S}_6,\hat{S}_6/{\sigma})\ .
\] This map agrees with the map defined by virtue of Proposition \ref{PrymMap}, i.e. with   $L\mapsto\CI_3(L)$ given by \eqref{I3L}. 
\end{proposition}

\begin{proof} Let $L\in Jac(S)$ be defined by divisor $D$ on $S$. With $\mathcal{C}$ as in the proof of Proposition \ref{PrymMap},  it follows from \eqref{Cx} and \eqref{CxR} that 
$$\mathcal{C}=Nm_2(\pi_1^{*}(D)\cap\Delta),$$
\noi and hence (as in the proof of Proposition \ref{PrymMap}) that $\mathcal{C}$ is in the linear system for $\CI_3(L)$.   It thus follows from Proposition \ref{PrymMap} that $C_{\Delta}$ defines a map between the indicated Prym varieties.\end{proof}
%
%
  \subsection{The restriction to the split real form $\SL(4,\R)$}  We now examine the consequences of imposing the additional condition $L^2=\mathcal{O}_S$ on the spectral data $(S,L)$  {described} in Section \ref{specral33}. 

\begin{proposition} Under the assumptions and notation of Proposition \ref{PrymMap}, if $L^2\simeq\mathcal{O}_S$, then $\mathcal{I}_3(L)$ is a point of order two in $Jac(\hat{S}_6)$.
\end{proposition}  

\begin{proof}  
The map $L\mapsto\mathcal{I}_3(L)$ defined by \eqref{I3L}  is a group homomorphism between Jacobians. Thus, in particular, if $L^2\simeq\mathcal{O}_S$ then $\mathcal{I}_3(L)^2=\mathcal{I}_3(L^2)=\mathcal{I}_3(\mathcal{O}_S)=\mathcal{O}_{\hat{S}_6}$.
\end{proof}

It follows (see \cite[Theorem 4.12]{thesis}) that our construction maps spectral data for $\SL(4,\R)$-Higgs bundles to spectral data for $\SO_0(3,3)$-Higgs bundles.
 As in the rank two case discussed in Section \ref{realform22},  if $\mathcal{I}_3(L)$ is a point of order two in $Prym(\hat{S}_6,\hat{S}_6/{\sigma})$ then it is invariant under the involution $\sigma$.  Under the projection $p:\hat{S}_6\rightarrow\hat{S}_6/{\sigma}$ the direct image sheaf $p_*\mathcal{I}_3L$ thus splits as the sum of two line bundles $\mathcal{I}_3(L)_{\pm}$,   generated by $\sigma$-invariant and anti-invariant local sections. 
\noi In particular,
\begin{equation}\label{decomp+-}
\hat{{{\pi}_6}}_*{\mathcal{I}_3(L)}={\pi_{\sigma}}_*\mathcal{I}_3(L)_+\oplus {\pi_{\sigma}}_*\mathcal{I}_3(L)_-.
\end{equation} We get a diagram similar to \eqref{diagramR}:
\begin{eqnarray}\label{diagramR4}
\xymatrix{
\mathcal{I}_3(L)\ar[d]&p_*({\mathcal{I}_3(L)})=\mathcal{I}_3(L)_+\oplus\mathcal{I}_3(L)_-\ar[d]\\
\hat{S}_6\ar[d]^{\hat{\pi}_6}\ar[r]^{p}&\hat{S}_6/{\sigma}\ar[dl]^{\pi_{\sigma}}\\
\Sigma&
}
\end{eqnarray}

  Recall from Proposition  {\ref{rank3main}} that $\hat{S}_6$ is the normalization of the spectral curve $S_6\subset\pi^*K$ and that the involution $\sigma$ on $\hat{S}_6$ corresponds to the involution $\eta\mapsto -\eta$ on $S_6$. It follows that multiplication by the tautological section $\eta$ interchanges $\mathcal{I}_3(L)_-$ and $\mathcal{I}_3(L)_+$, and thus that the Higgs field on $\hat{{{\pi}_6}}_*{\mathcal{I}_3(L)}$ has the form as in \eqref{SL4onHiggs}.  Combined with Proposition \ref{realI3}, we thus get: 

\begin{proposition} {Let}  $(S,L)$  {be} the spectral data for a point in $\mathcal{M}_{\SL(4,\R)}$ represented by  {a}  Higgs bundle $(E,\Phi)$ with orthogonal structure $q$ on $E${,} and isomorphism $\delta:\det E\simeq\mathcal{O}_{\Sigma}$.  Let  $(\hat{S}_6,\mathcal{I}_3(L))$ be defined as  \eqref{hatS6} and \eqref{I3L}, i.e.
$ \hat{S}_6:=Sym((S\times_{\Sigma}S)_0/{\tau})$ and $\mathcal{I}_3(L)=\mathcal{L}_-\otimes T^{-1}.$ 
\noi Then $(\hat{S}_6,\mathcal{I}_3(L))$  is the spectral data for the point in $\mathcal{M}_{\SO_0(3,3)}$ represented by 
\begin{equation}\label{theone}
\Bigg(\Lambda^2_+E\oplus\Lambda^2_-E, \left(\begin{array}{cc}0&\alpha\\-\alpha^{\rm{T}}&0\end{array}\right)\Bigg),
\end{equation}
\noi where the bundles have oriented orthogonal structures  $(q_{\pm},\delta_{\pm})$ as in \S \ref{Higgs33} {,}
and $\alpha$ is as in \eqref{SL4onHiggs}.
\end{proposition}




 \section{Maps between moduli spaces and Hitchin fibrations}\label{maps}
 
Thus far we have examined the maps induced by the isogenies on individual Higgs bundles or their spectral data. In this section we collect together some remarks about the induced maps on the corresponding moduli spaces and on their Hitchin fibrations. We note that  {whilst} the maps on Higgs bundles (given in \S \ref{iso22higgs} and Proposition \ref{realI3})  {are} defined for all Higgs bundles {,} but  {do} not obviously preserve stability properties. On the other hand, the maps on spectral data
(see  Propositions \ref{propso21} and \ref{rank3main}) automatically preserve stability but apply only to generic points in the moduli spaces  - where the stability condition is vacuous.  Pending further work we thus limit our remarks to the dense open sets in the moduli spaces which exclude the non-generic fibers of their Hitchin fibrations. We denote these sets by $\CMt_G\subset\mathcal{M}_G$.   

 {Note that for the groups in the isogenies studied in this paper,}  the special linear groups can be identified as the spin groups for the special orthogonal groups  {. T}he induced map $\CMt_{\mathrm{Spin}(2n,\C)}\rightarrow\CMt_{\SO(2n,\C)}$ is a finite map {.}  {As seen in \cite{N2}, f}or any  {$n$}, the moduli space $\mathcal{M}_{\SO(2n,\C)}$ has two components corresponding to the two possible values for the second Stieffel-Whitney class of an $\SO(2n,\C)$-principal bundle. In contrast, the underlying holomorphic bundles for the Higgs bundles in $\mathcal{M}_{\SL(2,\C)\times\SL(2,\C)}$ and $\mathcal{M}_{\SL(4,\C)}$ have just one topological type,  {and}  the moduli spaces are connected. The images of the maps $\CI_2$ and $\CI_3$ thus see just one of the components in $\mathcal{M}_{\SO(4,\C)}$ or $\mathcal{M}_{\SO(6,\C)}$. Indeed, this can be understood from the point of view of the surface group representations corresponding to the Higgs bundles via non-Abelian Hodge theory. From this point of view, the component in the image of the map contains precisely the representations in $\SO(2n,\C)$ which lift to $Spin(2n,C)$. We   therefore  see only the Higgs bundles in which the underlying holomorphic bundle has $w_2=0$.

The situation is more nuanced for the restriction of the maps to the moduli spaces for the split real forms. In this case, the underlying holomorphic bundles have more complicated topology than in the case of the Higgs bundles for the complex groups. Moreover,  fixing the topological type of the bundle does not ensure connectedness of the components. What remains true is that the images of the maps $\CI_2$ and $\CI_3$ contain only those components of the moduli spaces in which the Higgs bundles correspond to representations which lift to the appropriate spin group. 

 {In terms of Hitchin fibrations, one should note that for each $n=2,3$, t}he bases of the  fibrations of  {$Spin(2n,\C)$ and  $\SO(2n,\C)$-Higgs bundles} are the same {.}   {I}n the  case  {of $n=2$ the base is} $H^0(\Sigma,K^2)\oplus H^0(\Sigma,K^2)${,} and  {for $n=3$ it is} $H^0(\Sigma,K^2)\oplus H^0(\Sigma,K^3)\oplus H^0(\Sigma,K^4)$. In order to understand the maps induced on these bases it is necessary to understand exactly the relation between coordinates of a point in the base and the coefficients in the defining equation for the spectral curve. This is completely straightforward for $\SL(n,\C)$, where the two coincide, but less so in the case of $\SO(2n,\C)$ where the relation is complicated by the role of the Pfaffian.  {In fact,} the maps $S\mapsto \hat{S}_{2n}$ (for $n=2,3$) do not unambiguously descend to the base of the fibration. The ambiguity stems from the fact that the induced orthogonal structures on $E_1\otimes E_2$ or $\Lambda^2E$ do not have a canonical orientation. The choice of orientation corresponds on the one hand to a choice of trivialization on the determinant bundles, and on the other hand to a choice of sign in the Pfaffian.  

\begin{remark} {
One should note that whist not done here,  the isogenies could also be understood through the language of Cameral covers.}
\end{remark}

\subsection{The isogeny $\CI_2$ on moduli spaces} \label{iso22invariant} 
Once the isomorphism $\delta: E_1\otimes E_2\simeq\mathcal{O}_{\Sigma}$ is fixed, the map  \eqref{I*bundles} determines a map on the base of the Hitchin fibrations for $\mathcal{M}_{\SL(2,\C)\times\SL(2,\C)}$ and $\mathcal{M}_{\SO(4,\C)}$. The explicit form of the map depends on the generators chosen for the rings of invariant polynomials. Taking the coefficients of the characteristic equation $\det(\varphi-\eta I)$ as generators, it follows from \eqref{S4curve} that this map is given on the generic points in the base by
\begin{align}\label{basemap2}
\mathcal{I}_2: H^0(\Sigma,K^2)\oplus H^0(\Sigma,K^2)&\rightarrow H^0(\Sigma,K^2)\oplus H^0(\Sigma,K^2)\nonumber\\
(a_1,a_2)&\mapsto (2(a_1+a_2), \pm(a_1-a_2)),
\end{align}
\noi where the sign in the last component is determined by the isomorphism $\delta$.  For example, following the notation of Remark \ref{detiso}, if $\{e_1^1\otimes e^1_2,e_1^1\otimes e^2_2,e_1^2\otimes e^1_2,e_1^2\otimes e^2_2\}$ is the positively oriented local frame, then 
\begin{equation*}
\mathrm{Pf}(\Phi_1\otimes I+I\otimes\Phi_2)=a_2-a_1
\end{equation*}

\noi where $a_i=\det(\Phi_i)$, so the second component in \eqref{basemap2} is $-(a_1-a_2)$.

\begin{remark}\label{heuristic} As explained at the beginning of Section \ref{Lie}, the relation between spectral curves and thus the map \eqref{basemap2} {,} can be understood heuristically 
from the relation between the eigenvalues of the $\Phi_i$ and those of $\Phi_1\otimes I+I\otimes\Phi_2$. By Remark \ref{evalsadd}, at least at smooth unramified points, the equations for the spectral curves $S_i$ and $S_4$ are given by
\begin{align*}
S_i:\quad  0=&\eta^2+a_i=(\eta-\eta^{(i)}_1)(\eta+\eta^{(i)}_1),\  i=1,2\\
S_4:\quad 0=&\eta^4+b_2\eta^2+b_3^2=\prod(\eta\pm\eta^{(1)}_1\pm\eta^{(2)}_1)=(\eta^2-(\eta^{(1)}_1+\eta^{(2)}_1)^2)(\eta^2-(\eta^{(1)}_1-\eta^{(2)}_1)^2)\ ,
\end{align*}
\noi from which the relations between $(b_1,b_2,b_3^2)$ and $(a_1,a_2)$ can be deduced.

 \end{remark}
 
 As described in \S \ref{real so}, the vector bundle of an $SO_0(2,2)$-Higgs bundles may be expressed as $(M_1\oplus M_1^*)\oplus (M_2\oplus M_2^*)$, and thus {they are labelled by two integer topological invariants, say $(c_1,c_2)$, corresponding to the degrees of the line bundles $M_1$ and $M_2$ (or, equivalently by classes in $\pi_1(\SO(2,\C))\cong \Z$).} The moduli space $\CM_{\SO_0(2,2)}$ is thus a disjoint union of (possibly disconnected or empty) components
\begin{equation}
\CM_{\SO_0(2,2)}= \bigsqcup_{c_1,c_2} \CM_{\SO_0(2,2)}^{c_1c_2}.
\end{equation}
The moduli space  $\CM_{\SL(2,\R)\times\SL(2,\R)}$ is similarly a disjoint union of (possibly disconnected or empty) components 
\begin{equation}
\CM_{\SL(2,\R)\times\SL(2,\R)}= \bigsqcup_{d_1,d_2} \CM_{\SL(2,\R)\times\SL(2,\R)}^{d_1d_2},
\end{equation}
 {for $d_1,d_2$}  the degrees of the line bundles defining the $\SL(2,\R)\times\SL(2,\R)$-Higgs bundles.  As seen in \cite{N1}, the components are non-empty if and only if $|d_i|\leq g-1$, inequalities which are known as the the Milnor-Wood bounds.
 
\begin{proposition}\label{season2}
The map $\CI_2: \CMt_{\SL(2,\R)\times \SL(2,\R)}\rightarrow \CMt_{\SO_0(2,2)}$
\noi is  a $2^{2g+1}$-fold covering onto the components satisfying $c_1=c_2\ \mathrm{mod}\ 2$ and
$|c_i|\le 2g-2,$ for   $i=1,2$.  The map restricts to maps
\begin{equation}
\CI_2: \CMt_{\SL(2,\R)\times\SL(2,\R)}^{d_1d_2}\rightarrow  \CMt_{\SO_0(2,2)}^{c_1c_2}
\end{equation}
 {for} $c_1=d_1+d_2$ and $c_2=d_1-d_2$.

\end{proposition}
\begin{proof} The mod 2 congruence condition follows from the fact that any $\SO_0(2,2)$-Higgs bundle of the form \eqref{specsl2} in the image of $\CI_2$, has 
 \begin{align*}
 \deg(M_1)=&\deg(N_1)+\deg(N_2),\\
 \deg(M_2)=&\deg(N_1)-\deg(N_2),
 \end{align*}
 
 \noi where $N_1, N_2$ are line bundles defining $SL(2,\R)$-Higgs bundles (see Section \ref{iso22higgs}).  The bounds on $|c_i|$ follow from the Milnor-Wood bound for $\SL(2,\R)$-Higgs bundles.   The rest of the proposition follows from the observation that the preimage under $\CI_2$ for any Higgs bundle of the form \eqref{specsl2} consists of all $\SL(2,\R)\times \SL(2,\R)$-Higgs bundles defined by $(L_1,\beta_1,\gamma_1), (L_2,\beta_2,\gamma_2)$ with $ L_1^2=M_1M_2$ and $L_2^2=M_1M^{-1}_2$.   If $\deg(M_1)=\deg(M_2)\ \mathrm{mod}\  2$ then there are $2^{2g}$ solutions for $L_1$ and $L_2$.
 \end{proof}

  {Whilst the moduli space of}  $\SL(4,\R)${-Higgs bundles} has $2^{2g}$ Hitchin components,  {the one for} $\SO_0(3,3)${-Higgs bundles} has just one Hitchin component. From the analysis of topological invariants, which are constant on connected components,  {one has}   $4+1=5$ components,  4 coming  from the 4 pairs of $w_2$'s characterizing $\SO(3)\times\SO(3)$ bundles.

\begin{proposition}The  {isogeny $\CI_2$ between moduli spaces of Higgs bundles} takes the $2^{2g}$ Hitchin components to the one Hitchin component, and the other 2 components to the two components (possibly disconnected)  where the two $w_2$'s are the same.\label{hola}
\end{proposition}

  {As seen before,}  the map   $\CI_2$
  constructed in \S \ref{isosingular} is surjective onto  {some of the}   components of $\mathcal{M}_{\SO_0(2,2)}$. The components in the image  correspond to those components in the representation variety $Rep(\pi_1(\Sigma),\SO_0(2,2))$  {for}  which the representations lift   to $\SL(2,\R)\times\SL(2,\R)$. 
  We shall denote by $\mathcal{M}^0_{\SO_0(2,2)}$ the union of components of $\mathcal{M}_{\SO_0(2,2)}$ obtained through (\ref{I*bundles}) and (\ref{I*phi}) with 
$\deg(M_1)=\deg(M_2)\ \mathrm{mod}\ 2\ .$
  Equivalently,  let $Rep^0(\pi_1(\Sigma),\SO_0(2,2))$   be the union of components of   $Rep(\pi_1(\Sigma),\SO_0(2,2))$ which correspond to the components in $\mathcal{M}_0(\SO_0(2,2)$. 
 
\begin{corollary}\label{cor1} The structure group of an $\SO_0(2,2)$-Higgs bundle lifts to $\SL(2,\R)\times\SL(2,\R)$ if and only if the $\SO_0(2,2)$-Higgs bundle lies in $\mathcal{M}^0_{\SO_0(2,2)}$.  Equivalently, a reductive surface group representation into $\SO_0(2,2)$ lifts to a representation into $\SL(2,\R)\times\SL(2,\R)$ if and only if the representation lies $Rep^0(\pi_1(\Sigma),\SO_0(2,2))$.
\end{corollary}

\begin{remark}
By realising $SO_0(2,2)$-Higgs bundles in terms of rank 2-Higgs bundles, one can understand the monodromy action studied in \cite{mono} for rank 4 Higgs bundles in terms of monodromy of lower rank Higgs bundles. Indeed, taking  $b_1$ and $b_2$ as in Eq.~\eqref{S4curve} one recovers the rank 4 monodromy as a product of actions coming from the rank 2-Hitchin systems. 
\end{remark}

\subsection{The isogeny $\CI_3$ on moduli spaces}\label{inv33}
 
To fully specify the map induced by $\CI_3$ (as in \eqref{I3onHiggs}) on $\mathcal{M}_{\SL(4,\C)}$  {one needs} to specify the trivialization $\delta:\det(\Lambda^2(E))\simeq\mathcal{O}_{\Sigma}$.  Then, using the coefficients of the  characteristic equation $\det(\varphi-\eta I)$ as generators for the rings of invariant polynomials, it follows from \eqref{bvsa} that the map on the generic points   {of the Hitchin base is} \begin{align}\label{basemap3}
\mathcal{I}_3: H^0(\Sigma,K^2)\oplus H^0(\Sigma,K^3)\oplus  H^0(\Sigma,K^4)&\rightarrow H^0(\Sigma,K^2)\oplus H^0(\Sigma,K^3)\oplus  H^0(\Sigma,K^4)\nonumber\\
(a_2,a_3, a_4)&\mapsto (2a_2, \pm a_3, a_2^2-4a_2),
\end{align}
\noi where the sign in the second component is determined by the isomorphism $\delta$. 

\begin{remark}
For both  {n=2,3}  the  {variety $Prym(\hat{S}_{2n}, \hat{S}_{2n}/_{\sigma}))$}  has two components. In the case  {$n=3$},  {given} $\hat{S}_6$   the spectral curve defined by $(2a_2, a_2^2-4a_2, a_3^2)$,    {the two}  components occur in the fibers over both points $(2a_2, \pm a_3, a_2^2-4a_2)$ in the    Hitchin  {base}.  {However, o}n each fiber only one of the components is in the image of the map induced by $\CI_3$.  A similar phenomenon occurs for $n=2$.
\end{remark}

\begin{remark} The map \eqref{basemap3} can be understood heuristically from eigenvalue considerations in the same way as \eqref{basemap2}, i.e. as explained in Remark \ref{heuristic}  \end{remark}

As a consequence of its special orthogonal structure, the vector bundle in an $\SL(4,\R)$-Higgs bundle is classified topologically by a second Stiefel-Whitney class.  The moduli space $\CM_{SL(4,\R)}$ thus decomposes into components (not necessarily connected) labelled by this $\mathbb{Z}_2$-valued invariant. The moduli space $\CM_{SO_0(3,3)}$ likewise has a decomposition into components labelled by a pair of $\mathbb{Z}_2$-valued invariants corresponding to the second Stiefel-Whitney classes of the two rank three bundles.  Using superscripts to indicate these characteristic classes, we can thus write
\begin{eqnarray}
\mathcal{M}_{\SL(4,\R)} =\coprod_{a~\in~ H^2(\Sigma,\Z_2)}\mathcal{M}^a_{\SL(4,\R)},~{~\rm~and~}~
\mathcal{M}_{\SO_0(3,3)}=\coprod_{b_i~\in ~H^2(\Sigma,\Z_2)}\mathcal{M}^{(b_1,b_2)}_{\SO_0(3,3)}.\nonumber
\end{eqnarray}

\begin{proposition}\label{propinv1}
The map $\CI_3 :\CMt_{\SL(4,\R)}\rightarrow \CMt_{\SO_0(3,3)}$
\noi is  a $2^{2g}$-fold covering onto the components satisfying $b_1=b_2\ \mathrm{mod}\ 2$.  {For $a=w_2(E)$   the orthogonal rank 4 bundle $E$, and $b=w_2(\Lambda^2(E)$, t}he map restricts to   
\begin{equation}\label{oncomponents}
\CI_3 :\CMt^a_{\SL(4,\R)}\rightarrow \CMt^{(b,b)}_{\SO_0(3,3)}.
\end{equation}

\end{proposition}

\begin{proof} It follows from Proposition \ref{realI3} that the image of $\CMt^a_{\SL(4,\R)}$ lies in  {the} component $\CMt^{(b_1,b_2)}_{\SO_0(3,3)}${,} where $b_1=w_2(\Lambda^2_+(E))$ and $b_2=w_2(\Lambda^2_-(E))$ for some $\SO(4,\C)$ vector bundle $E$ with $w_2(E)=a$.   {Moreover, from \cite[Proposition 1.8]{GZ} it follows that}    $w_2(\Lambda^2(E)_{\pm})=w_2(\Lambda^2(E))$, so in particular $b_1=b_2$. However, any pair of $\SO(3,\C)$-bundles with the same second Stieffel-Whitney class arises in this way, i.e.. as $\Lambda^2(E)_{\pm}$ where $E$ is an $\SO(4,\C)$ bundle. 
 {Finally, i}f $(\Lambda^2E, \Phi\otimes I+I\otimes\Phi)$ represents a point in the image of $\CI_3$, then the $2^{2g}$ preimages come from twisting $E$ by any point of order two in $Jac(\Sigma)$.  
\end{proof}

\end{document}